\documentclass[a4paper]{article}

\usepackage{amsmath} 

\usepackage{amsthm} 

\usepackage{amssymb} 
\usepackage[shortcuts]{extdash} 
\usepackage{floatrow}
\usepackage{graphicx}
\usepackage{subfig}

\usepackage{mathrsfs} 

\usepackage{color}

\newcommand{\beq}{\begin{equation}}
\newcommand{\eeq}{\end{equation}}
\newcommand{\beqa}{\begin{eqnarray}}
\newcommand{\eeqa}{\end{eqnarray}}

\newtheorem{defi}{Definition}
\newtheorem{lemma}{Lemma}
\newtheorem{prop}{Proposition}
\newtheorem{remark}{Remark}
\newtheorem{thm}{Theorem}

\newtheorem*{npGL}{New properties}

\def\hmu{\mbox{$\hat{\mu}$}}
\def\oD{\mbox{$\overline{\dd}$}}
\def\of{\mbox{$\overline{f}$}}
\def\ofprime{\mbox{$\overline{f'}$}}
\def\ofs{\mbox{\scriptsize $\overline{f}$}}
\def\oG{\mbox{$\overline{G}$}}
\def\ovarphi{\mbox{$\overline{\varphi}$}}
\def\PO{\mbox{$\mathbb P$}}

\def\hnm{\mbox{$h_{n-}$}}
\def\hnp{\mbox{$h_{n+}$}}

\def\hnme#1{\mbox{$h_{n-}^{#1}$}}
\def\hnpe#1{\mbox{$h_{n+}^{#1}$}}

\def\dst#1{\mbox{$d^*_{#1}$}}

\def\hst#1{\mbox{$h^*_{#1}$}}

\def\hnm{\mbox{$h_{n-}$}}
\def\hnp{\mbox{$h_{n+}$}}
\def\hsnm{\mbox{$h^*_{n-}$}}
\def\hsnp{\mbox{$h^*_{n+}$}}

\def\tep{\mbox{$\tilde{\varepsilon}$}}
\def\tiM{\mbox{$\widetilde{M}$}}

\def\xinm{\mbox{$\xi_{n-}$}}
\def\xinp{\mbox{$\xi_{n+}$}}

\def\dd{\mbox{$\mathscr D$}}

\def\NA{\mathbb N}

\def\RE{\mathbb R}

\begin{document}
\title{\bf Analysis of difference schemes for the Fokker\=/Planck angular diffusion operator}
\author{
\'Oscar L\'opez Pouso\thanks{Department of Applied Mathematics, Faculty of Mathematics, University of Santiago de
Com\-pos\-te\-la, Santiago de Com\-pos\-te\-la (A Co\-ru\-\~na), Spain. Email: {\tt oscar.lopez@usc.es}.}\\
Javier Segura\thanks{Departamento de Matem\'aticas, Estad\'{\i}stica y Computaci\'on. Universidad de Cantabria, San\-tan\-der, Spain. Email: {\tt javier.segura@unican.es}.}}

\maketitle

\abstract{This paper is dedicated to the mathematical analysis of finite difference schemes for the angular diffusion operator present in the azimuth\=/independent Fokker\=/Planck equation. The study elucidates the reasons behind the lack of convergence in half range mode for certain widely recognized discrete ordinates methods, and establishes sets of sufficient conditions to ensure that the schemes achieve convergence of order $2$. In the process, interesting properties regarding Gaussian nodes and weights, which until now have remained unnoticed by mathematicians, naturally emerge.}

\vspace{0.4cm}





\section{Introduction}\label{intro}
The following acronyms will be used:
\begin{itemize}
\item DOM: discrete ordinates method.
\item FP, FPE: Fokker\=/Planck, Fokker\=/Planck equation.
\item GL: Gauss\=/Legendre (quadrature rule in $(-1,1)$).
\item PDE: partial differential equation.
\end{itemize}

\par This paper focuses on analyzing difference schemes that discretize the FP angular diffusion operator in the azimuth\=/independent case
\begin{equation}
\label{AICSO}
\Delta_{\rm FP} f(\mu) = (\dd(\mu) f'(\mu))',\quad \mu \in [-1,1],
\end{equation}
where $\dd(\mu) = 1 - \mu^2$.

\par This operator is important because it is a fundamental part of the FPE. In turn, the FPE is a forward\=/backward parabolic PDE, highly significant in the field of nuclear engineering, in which $f$ represents the angular flux of particles, while $\mu$, which is the cosine of the polar angle, determines the direction of particle propagation. Interested readers can refer to various references, including \cite{GALP23}, \cite{LPJU16}, or \cite{LPJU21}, to delve deeper into this topic. The term FP angular diffusion operator is also known by other names such as {\em continuous scattering operator}, {\em FP Laplacian}, {\em Laplacian on the unit sphere}, {\em spherical Laplacian}, or {\em Laplace-Beltrami operator}.

\par A commonly employed technique for solving the FPE is the use of a DOM, which discretizes the operator (\ref{AICSO}) by utilizing a suitably selected set of nodes. Although various choices are possible, a frequently adopted approach is to use the GL nodes. In this paper, DOM discretizations that use GL nodes will be referred to as {\em GL schemes}.

\par This work originated with the primary intention of carrying out a mathematical analysis of the GL scheme proposed by Morel in \cite{MO85}. To conduct this analysis, it has been valuable for us to define two categories of schemes referred to as {\em type I} and {\em type II}. Morel's scheme belongs to the type II category, whereas type I schemes encompass two other well\=/known DOMs which are again GL schemes: the one employed by Antal, Lee, Mehlhorn, and Duderstadt in \cite{ANLE76}, \cite{LE62} and \cite{MEDU80}, and the one utilized by Haldy and Ligou in \cite{HALI80}.

\par The main objective is to establish the convergence of these schemes with second\=/order accuracy. While addressing this problem is relatively straightforward when considering uniform meshes, it becomes significantly more challenging when the nodes are not equally spaced, such as in the case of GL schemes. Type II schemes present an additional difficulty in that they deviate from the conventional formulation of numerical differentiation formulas. This is because they do not use exact values of $\dd$, but rather convenient approximations.

\par The present work focuses on studying discretizations of the operator (\ref{AICSO}) in isolation, which allows bringing to light the characteristics of the approximations and carrying out a clearer analysis of them.

\par We notice that the computing power of current PCs, together with recent research that allows the calculation of nodes and weights of GL formulas with millions of nodes in a few seconds of laptop time (see \cite{GISETE21} and references therein), makes it possible to program GL schemes without too much cost even when the number of nodes is large.

\par For the purposes of this study, the term {\em diffusivity} will be used to refer to $\dd$, recognizing that this decision entails some linguistic flexibility, given that $\dd$ originates from the mathematical expression of the spherical Laplacian and does not directly represent any physical property of the medium.

\par Many of the ideas presented herein can also be used if $\dd(\mu)$ is different from $1 - \mu^2$, as long as it satisfies some natural conditions.

\par After the elementary remainder that the reader will find in Section \ref{SECTION-elementary}, this paper is structured as follows:
\begin{itemize}
\item Section \ref{SECTION-mesh} focuses on defining the specific type of meshes considered in the paper and on setting the properties they must satisfy.

\item In Section \ref{SECTION-properties-GL}, we review established properties of GL nodes and weights, while also presenting novel properties discovered during the study of the schemes in this article. These additional properties play a crucial role in proving that some important schemes converge with order $2$.

\item Section \ref{SECTION-underlying} comprises two lemmas that serve as the foundation for proving the main results in subsequent sections.

\item Section \ref{SECTION-general-comments} explains the concepts of convergence of order $p$, full and half range mode, and preservation of moments.

\item Sections \ref{SECTION-type-I} and \ref{SECTION-type-II} form the core of the paper, providing a detailed description and analysis of type I and type II schemes, respectively, accompanied by numerical results.

\item Section \ref{SECTION-conclusions} finishes the paper by summarizing the findings and drawing overall conclusions.
\end{itemize}

\section{An elementary reminder}\label{SECTION-elementary}
Let $\mu$ be an interior point, i.e., $\mu \in (-1,1)$ and let us understand that, for a general function $G$ and small $h > 0$, $\oG_s = G(\mu + sh)$.

\par It will be useful to keep in mind that the classical formula
\begin{equation}
\label{interior-classical} \Delta_{\rm FP} f(\mu) \approx \frac{\oD_{-1/2} \of_{-1} - (\oD_{-1/2} + \oD_{1/2}) \of_0 + \oD_{1/2} \of_1}{h^2}
\end{equation}
can be interpreted as the outcome of repeatedly applying, with step\=/size $h/2$, the centered formula for the first derivative:
\begin{equation}
\label{first-derivative-c}
\varphi'(\mu) = \frac{\ovarphi_1 - \ovarphi_{-1}}{2h} + E(h).
\end{equation}

\par Indeed, (\ref{interior-classical}) follows from
\begin{equation}
\label{i-classical-deduction}
\Delta_{\rm FP} f(\mu) \approx \frac{\oD_{1/2} \ofprime_{1/2} - \oD_{-1/2} \ofprime_{-1/2}}{h} \approx \frac{\oD_{1/2} \frac{\ofs_1 - \ofs_0}{h} - \oD_{-1/2} \frac{\ofs_0 - \ofs_{-1}}{h}}{h}.
\end{equation}

\par If $\varphi \in {\rm C}^3([-1,1])$, the formula (\ref{first-derivative-c}) achieves order $2$, i.e., $E(h) = O(h^2)$. However, one cannot infer from this property that the formula (\ref{interior-classical}) also possesses second-order accuracy. This is because the presence of $h$ in the denominator of the last fraction in Equation (\ref{i-classical-deduction}) could make the order decay down to $1$. Fortunately, this undesired effect does not occur, and the following theorem holds. The proof, which relies on Taylor expansions, is omitted here since this is a well\=/established result.

\begin{thm}\label{thm-interior-classical}
If $f \in {\rm C}^4([-1,1])$, then the differentiation formula (\ref{interior-classical}) has order $2$, i.e.,
\begin{equation}
\Delta_{\rm FP} f(\mu) = \frac{\oD_{-1/2} \of_{-1} - (\oD_{-1/2} + \oD_{1/2}) \of_0 + \oD_{1/2} \of_1}{h^2} + O(h^2).
\end{equation}
\end{thm}

\begin{remark}
Theorem \ref{thm-interior-classical} still holds if $\dd$ is replaced by any other diffusivity, as long as it belongs to ${\rm C}^3([-1,1])$.
\end{remark}

\par The differentiation formula (\ref{interior-classical}) can be applied at the interior points of a uniform mesh of $[-1,1]$ in a quite obvious way. Since, as said above, the GL nodes are not equally spaced, a broader framework is needed, and this will be the focus of the next sections.

\section{The mesh}\label{SECTION-mesh}
Considering the influence of the schemes utilized in nuclear engineering that served as a motivation for this work, we will focus exclusively on meshes comprising interior nodes. While, as exemplified in \cite{LPJU16}, it is feasible to devise schemes that incorporate $-1$ and $1$ as nodes, the study of such cases will be deferred for future research.

\par Specifically, we will consider several instances of the following situation: for every natural $N$, we want to approximate the operator (\ref{AICSO}) on a mesh of $N$ nodes $\mu_1^N,\dots,\mu_N^N$, located in the open interval $(-1,1)$ and not necessarily equally spaced,  with the aid of an auxiliary set of $N + 1$ points $\mu_{1/2}^N,\dots,\mu_{N+1/2}^N$, also not necessarily equally spaced. Note the difference in meaning between `node' and `point.'

\par The sets of nodes and points are supposed to be interlaced conforming to the following pattern:
\begin{equation}
\label{interlacing}
-1 = \mu_{1/2}^N < \mu_1^N < \mu_{1+1/2}^N < \cdots < \mu_{N-1/2}^N < \mu_N^N < \mu_{N+1/2}^N = 1.
\end{equation}

\begin{defi}
$M_N$ and $\tiM_N$ are the numbers defined by
\begin{align}
M_N = \max_{1 \leq n \leq N-1}\{\mu_{n+1}^N - \mu_n^N\},\\
\tiM_N = \max\{\mu_1^N + 1, M_N, 1 - \mu_N^N\}.
\end{align}
\end{defi}

\par The minimum requirement for $[\{\mu_n^N\}_{n=1}^N$, $N \in \NA]$ to be considered a collection of meshes of $[-1,1]$ is that
\begin{equation}\label{mesh-lim-is-0}
\lim_{N \to \infty} \tiM_N = 0,
\end{equation}
but here a stronger assumption is needed, namely that
\begin{equation}
\label{mesh-order-is-1}
\tiM_N = O(N^{-1}),
\end{equation}
as it happens for uniform meshes.

\begin{remark}
\label{rem-mesh-order-is-1}
Since $\mu_1^N + 1 + \sum_{n=1}^{N-1} (\mu_{n+1}^N - \mu_n^N) + 1 - \mu_N^N = 2$ by (\ref{interlacing}), it is sure that $2 \leq (N + 1)\tiM_N$, which in turn implies that it is impossible to have $\tiM_N = O(N^{-p})$ with $p > 1$. However, $\tiM_N$ could potentially be a $O(N^{-p})$ with $p \in (0,1)$ if one assumes only (\ref{interlacing}) and (\ref{mesh-lim-is-0}).
\end{remark}

\begin{remark}\label{remarkGLnodes}
According to Remark \ref{rem-mesh-order-is-1}, with (\ref{mesh-order-is-1}) we are supposing that the elements of $\{-1,\mbox{nodes},1\}$ are as close together as they can be, but this does not prevent the order $1$ from being exceeded locally; for example, GL nodes satisfy (\ref{mesh-order-is-1}) and accumulate quadratically at the end\=/points of $(-1,1)$; other examples can be furnished by applying appropriate functions to the nodes of a uniform mesh.
\end{remark}

\par It is clear that (\ref{mesh-order-is-1}) implies that
\begin{equation}
\label{nodes-order-is-1}
M_N = O(N^{-1}).
\end{equation}

\par The scheme (\ref{interior-classical}) can be easily adapted to this more general situation, and, naturally, we would like to get conditions which make the new scheme to have order $2$. Recalling Section \ref{SECTION-elementary}, one can correctly intuit in this regard that the hypotheses (\ref{interlacing}) and (\ref{mesh-order-is-1}) will not be enough, because $\mu_n$ and $\mu_{n+1/2}$ are not necessarily located at the center of the cells $[\mu_{n-1/2},\mu_{n+1/2}]$ and $[\mu_n,\mu_{n+1}]$. What may be less apparent is that these hypotheses not only fail to guarantee second\=/order convergence, but they are also insufficient to ensure mere convergence. Later we will prove that everything unfolds smoothly if $\mu_n$ and $\mu_{n+1/2}$ are {\em sufficiently close} to the mentioned central points as long as several appropriate assumptions are added to the picture.

\par Accordingly, we proceed by introducing a set of new conditions that build upon the existing hypotheses (\ref{interlacing}) and (\ref{mesh-order-is-1}), bringing us closer to the desired objective.

\begin{defi}\label{points-norm}
$M_N^* = \max_{1 \leq n \leq N}\{\mu_{n + 1/2}^N - \mu_{n-1/2}^N\}$.
\end{defi}

\par Since the elements of $\{-1,\mbox{nodes},1\}$ are supposed to be as close together as they can be, the hypothesis (\ref{interlacing}) implies that the same will happen to the points, that is,
\begin{equation}
\label{points-order-is-1}
M_N^* = O(N^{-1}).
\end{equation}

\par More precisely, the following lemma holds.

\begin{lemma}
\label{lemma-mesh-points-order-is-1}
Under the hypothesis (\ref{interlacing}), conditions (\ref{mesh-order-is-1}) and (\ref{points-order-is-1}) are equivalent.
\end{lemma}

\begin{proof}
Simply notice that (\ref{mesh-order-is-1}) implies (\ref{points-order-is-1}) because $M_N^* \leq 2 \tiM_N$ and (\ref{points-order-is-1}) implies (\ref{mesh-order-is-1}) because $\tiM_N \leq 2 M_N^*$. Both inequalities are readily delivered from hypothesis (\ref{interlacing}).
\end{proof}

\begin{defi}\label{def-secondary}
The set of secondary nodes $\{\hmu_n^N\}_{n = 1}^N$ is defined as follows:
\begin{equation}
\hmu_n^N = (\mu_{n-1/2}^N + \mu_{n+1/2}^N)/2,
\end{equation}
i.e., $\hmu_n^N$ is the mid\=/point of the cell $[\mu_{n-1/2}^N,\mu_{n+1/2}^N]$.

\par The set of secondary points $\{\hmu_{n+1/2}^N\}_{n = 1}^{N-1}$ is defined as follows:
\begin{equation}
\hmu_{n+1/2}^N = (\mu_n^N + \mu_{n+1}^N)/2,
\end{equation}
i.e., $\hmu_{n+1/2}^N$ is the mid\=/point of the cell $[\mu_n^N,\mu_{n+1}^N]$.
\end{defi}

\begin{defi}\label{DNstar}
$D_N^\star = \max_{1 \leq n \leq N}|\hmu_n^N - \mu_n^N|$.
\end{defi}

\begin{defi}\label{DN}
$D_N = \max_{1 \leq n \leq N-1}|\hmu_{n+1/2}^N - \mu_{n+1/2}^N|$.
\end{defi}

\par The following result holds.

\begin{lemma}
\label{lemma-mid-points}
Under the hypotheses (\ref{interlacing}) and (\ref{mesh-order-is-1}), there exist $q \geq 1$ and $r \geq 1$ such that
\begin{align}
\label{mid-point-1} D_N^\star & = O(N^{-q}),\\
\label{mid-point-2} D_N & =  O(N^{-r}).
\end{align}
\end{lemma}

\begin{proof}
Due to (\ref{interlacing}), it is sure that $\hmu_n^N, \mu_n^N \in (\mu_{n-1/2}^N,\mu_{n+1/2}^N)\ \mbox{for}\ n = 1,\dots,N$. So,
\begin{equation}
\max_{1 \leq n \leq N}|\hmu_n^N - \mu_n^N| \leq M_N^*.
\end{equation}
Now (\ref{mid-point-1}) is implied by (\ref{points-order-is-1}).

\par Analogously, (\ref{mid-point-2}) is implied by the inequality
\begin{equation}
\max_{1 \leq n \leq N-1}|\hmu_{n+1/2}^N - \mu_{n+1/2}^N| \leq M_N
\end{equation}
and (\ref{nodes-order-is-1}).
\end{proof}

\par However, $D_N^\star = O(N^{-1})$ and $D_N =  O(N^{-1})$ are not enough for ensuring quadratic convergence. To achieve this goal, we will make the assumption that both $q$ and $r$ in (\ref{mid-point-1}) and (\ref{mid-point-2}) are not less than $2$:
\begin{align}
\label{mid-point-1-hyp} D_N^\star & = O(N^{-q})\ \mbox{with}\ q \geq 2,\\
\label{mid-point-2-hyp} D_N & =  O(N^{-r})\ \mbox{with}\ r \geq 2.
\end{align}

\begin{defi}\label{min-points}
$m_N^* = \min_{1 \leq n \leq N}\{\mu_{n + 1/2}^N - \mu_{n-1/2}^N\}$.
\end{defi}

\par The hypotheses that we have enunciated so far are necessary to have convergence of order $2$. On the contrary, there are signs that the one that comes now could be weakened if $f$ were regular enough. It is not very restrictive though, and simplifies the proofs that will come later. Specifically, it will be assumed that
\begin{equation}
\label{hyp-min-points}
\frac{1}{m_N^*} = O(N^s)\ \mbox{with}\ 1 \leq s \leq 4m - 2,\ \mbox{where}\ m = \min\{q,r\}.
\end{equation}

\begin{remark}
Notice that $s < 1$ is impossible because the trivial equality $\sum_{n=1}^{N} (\mu_{n + 1/2}^N - \mu_{n-1/2}^N) = 2$ implies that $1/m_N^* \geq N/2$. The upper bound $4m - 2$ prevents $m_N^*$ from decreasing too fast, but the rate of decrease could still be considerably high, since $4 m - 2 \geq 6$. This is why we say above that this hypothesis is not very restrictive.
\end{remark}

\par The following lemma will be useful. Its proof is simple from Definition \ref{def-secondary} and is omitted.

\begin{lemma}
\label{lemma-mean-value}
For $n = 2,\dots,N-1$,
\begin{equation}\label{eq-lemma-mean-value}
\frac{\mu_{n-1}^N + \mu_{n+1}^N}{2} - \mu_n^N = (\hmu_{n-1/2}^N - \mu_{n-1/2}^N) + (\hmu_{n+1/2}^N - \mu_{n+1/2}^N) + 2(\hmu_n^N - \mu_n^N).
\end{equation}

\par Hence, under the hypotheses (\ref{mid-point-1-hyp}) and (\ref{mid-point-2-hyp}),
\begin{equation}
\max_{2 \leq n \leq N-1}\left| \frac{\mu_{n-1}^N + \mu_{n+1}^N}{2} - \mu_n^N \right| = O(N^{-m}),\ \mbox{with}\ m = \min\{q,r\} \geq 2.
\end{equation}
\end{lemma}

\section{Properties of GL nodes and weights}\label{SECTION-properties-GL}
Here we collect a brief list of facts about GL quadrature that will be needed later. Symmetry of weights and antisymmetry of nodes with respect to $0$ are assumed to be known.

\par The following result expresses in a formal way what was said about GL nodes in Remark \ref{remarkGLnodes}. Whenever GL nodes are mentioned, it must be understood that they are arranged in increasing order.

\begin{prop}\label{GLnodesOK}
If $\{\mu_n^N\}_{n=1}^N$ are the GL nodes, then the following assertions, where the exponents $1$ and $2$ are optimal, hold:
\begin{enumerate}
\item[(A)] $\tiM_N = O(N^{-1})$, that is, hypothesis (\ref{mesh-order-is-1}) holds.

\item[(B)] For any fixed natural $k$,
\begin{equation}
0 < \mu_1^N + 1 < \mu_2^N - \mu_1^N < \cdots < \mu_k^N - \mu_{k-1}^N\quad \mbox{if}\quad N \geq 2k
\end{equation}
and
\begin{equation}
\mu_k^N + 1 = 1 - \mu_{N-k+1}^N = O(N^{-2}).
\end{equation}
\end{enumerate}
\end{prop}

\begin{remark}[meaning of `optimal exponent']
An equivalent way of saying that the exponent $1$ is optimal in the expression $\tiM_N = O(N^{-1})$ is to say that $\tiM_N = \Theta(N^{-1})$ (`Big Theta' of $N^{-1}$). Similarly, $\mu_k^N + 1 = \Theta(N^{-2})$.
\end{remark}

\begin{prop}\label{GLweightsOK}
Let $\{w_n^N\}_{n=1}^N$ be the set of GL weights and let $k$ be any fixed natural number. Then
\begin{equation}
0 < w_1^N < w_2^N < \cdots < w_k^N\quad \mbox{if}\quad N \geq 2k
\end{equation}
and
\begin{equation}
w_k^N = O(N^{-2}),
\end{equation}
being the exponent $2$ optimal.

\par Therefore,
\begin{equation}
\frac{1}{\min_{1 \leq n \leq N} w_n^N} = \max_{1 \leq n \leq N} \frac{1}{w_n^N} = \frac{1}{w_1^N} = \frac{1}{w_N^N} = O(N^2).
\end{equation}
\end{prop}

\paragraph{\bf Other properties of GL nodes and weights.} The statements in Propositions \ref{GLnodesOK} and \ref{GLweightsOK} are established facts (proofs can be derived from results in \cite{SZE75}), but, as far as we know, the properties that follow are new. We have become aware of them since they are inherent to schemes used in nuclear engineering such as Haldy\=/Ligou's or Morel's, to be described later. The reference \cite{LPSE23} gathers strong theoretical evidence that they are true, and all numerical experiments that we have carried out corroborate them, but as yet there is no complete mathematical proof available.

\begin{npGL}
Let us suppose that $\{\mu_n^N\}_{n=1}^N$ and $\{w_n^N\}_{n=1}^N$ are, respectively, the GL nodes and weights, and that the points $\{\mu_{n+1/2}^N\}_{n=0}^N$ are defined by
\begin{align}
& \mu_{1/2}^N = -1,\\
& \mu_{n+1/2}^N = \mu_{n-1/2}^N + w_n^N\ \mbox{for}\ n = 1,\dots,N.
\end{align}
Then
\begin{itemize}
\item The hypothesis (\ref{interlacing}) holds, and

\item The hypotheses (\ref{mid-point-1-hyp}) and (\ref{mid-point-2-hyp}) are met with $q = r = 2$, that is, $D_N^* = O(N^{-2})$ and $D_N = O(N^{-2})$. Accordingly, by Lemma \ref{lemma-mean-value},
\begin{equation}
\max_{2 \leq n \leq N-1}\left| \frac{\mu_{n-1}^N + \mu_{n+1}^N}{2} - \mu_n^N \right| = O(N^{-2}).
\end{equation}
\end{itemize}
\end{npGL}

\par For the sake of ease, the superscript $N$ will be omitted in what follows.

\section{The underlying formulas}\label{SECTION-underlying}
\begin{defi}
$\PO_k$, with $k \in \NA$, will be the real vector space of all polynomials with real coefficients having degree less than or equal to $k$.
\end{defi}

\begin{defi}[quantities of interest related to cell {$[\mu_n,\mu_{n+1}]$}]
\label{defi-h-and-d}
For $n = 1,\dots,N-1$:
\begin{align}
h_n & = (\mu_{n+1} - \mu_n)/2,\\
d_n & = \hmu_{n+1/2} - \mu_{n+1/2},\\
\hnm & = \mu_{n+1/2} - \mu_n,\\
\hnp & = \mu_{n+1} - \mu_{n+1/2}.
\end{align}
\end{defi}

\begin{defi}[quantities of interest related to cell {$[\mu_{n-1/2},\mu_{n+1/2}]$}]
\label{defi-h-and-d-star}
For $n = 1,\dots,N$:
\begin{align}
\hst{n} & = (\mu_{n+1/2} - \mu_{n-1/2})/2,\\
\dst{n} & = \hmu_n - \mu_n,\\
\hsnm & = \mu_n - \mu_{n-1/2},\\
\hsnp & = \mu_{n+1/2} - \mu_n.
\end{align}
\end{defi}

\begin{remark}
It is obvious that $\hsnm = h_{(n-1)+}$ if $n \in \{2,\dots,N\}$, and that $\hsnp = \hnm$ if $n \in \{1,\dots,N-1\}$.
\end{remark}

\begin{remark}
Due to (\ref{interlacing}), $h_n$, $\hnm$, $\hnp$, $\hst{n}$, $\hsnm$ and $\hsnp$ are always positive. On the other hand, $d_n$ and $\dst{n}$ can be positive, negative, or zero.
\end{remark}

\par The above Definitions \ref{defi-h-and-d} and \ref{defi-h-and-d-star} imply that, for $n = 1,\dots,N-1$,
\begin{align}
\hnm & = h_n - d_n,\\
\hnp & = h_n + d_n,\\
\hnm + \hnp & = 2 h_n = \mu_{n + 1} - \mu_n,\\
\hnp - \hnm & = 2 d_n,\\
\label{ds-and-hs} d_n + \dst{n} & = h_n - \hst{n},
\end{align}
for $n = 2,\dots,N$,
\begin{equation}
\label{ds-and-hs-BIS} d_{n-1} + \dst{n} = \hst{n} - h_{n-1},
\end{equation}
for $n = 2,\dots,N-1$,
\begin{align}
\label{odd} d_{n-1} - d_n = 2\hst{n} - (h_{n-1} + h_n),
\end{align}
and, for $n = 1,\dots,N$,
\begin{align}
\hsnm & = \hst{n} - \dst{n},\\
\hsnp & = \hst{n} + \dst{n},\\
\hsnm + \hsnp & = 2 \hst{n} = \mu_{n + 1/2} - \mu_{n - 1/2},\\
\hsnp - \hsnm & = 2 \dst{n}.
\end{align}

\par Also,
\begin{align}
M_N & = 2 \max_{1 \leq n \leq N-1} h_n,\\
M_N^* & = 2 \max_{1 \leq n \leq N} \hst{n},\\
D_N & = \max_{1 \leq n \leq N-1} |d_n|,\\
D_N^* & = \max_{1 \leq n \leq N} |\dst{n}|,\\
m_N^* & = 2 \min_{1 \leq n \leq N} \hst{n}.
\end{align}

\par In light of Section \ref{SECTION-elementary}, we will exploit the following two lemmas. We will use the notation $\| \psi \|_\infty = \max_{\mu \in [-1,1]} |\psi(\mu)|$, understanding that $\psi \in {\rm C}([-1,1])$. Also, the notations $\xi_{n-}^*$, $\xi_{n+}^*$, $\xinm$, $\xinp$ will stand for intermediate values appearing in the Lagrange form of the Taylor remainder. It will be important to bear in mind that $x^k + y^k$ is divisible by $x + y$ when $k$ is odd, and that $x^k - y^k$ is divisible by $x + y$ when $k$ is even.

\par
\begin{lemma}
\label{mother-formula-1}
Assume that the hypothesis (\ref{interlacing}) holds.

\par The approximation
\begin{equation}
\label{seed-1}
\varphi'(\mu_n) \approx \frac{\varphi(\mu_{n+1/2}) - \varphi(\mu_{n-1/2})}{\mu_{n+1/2} - \mu_{n-1/2}},\ n = 1,\dots,N,
\end{equation}
converges with order $2$ if, and only if, the hypotheses (\ref{mesh-order-is-1}) and (\ref{mid-point-1-hyp}) are met. More precisely, if $E_n^*(\varphi)$ is defined by
\begin{equation}
\label{EnstarDEF}
E_n^*(\varphi) = \varphi'(\mu_n) - \frac{\varphi(\mu_{n+1/2}) - \varphi(\mu_{n-1/2})}{\mu_{n+1/2} - \mu_{n-1/2}},\ n = 1,\dots,N,
\end{equation}
then
\begin{equation}
\max_{1 \leq n \leq N} |E_n^*(\varphi)| = O(N^{-2})\ \mbox{for all}\ \varphi \in {\rm C}^3([-1,1])
\end{equation}
if, and only if, the hypotheses (\ref{mesh-order-is-1}) and (\ref{mid-point-1-hyp}) are met.

\par The maximal possible order is $2$.

\par Moreover, the formula (\ref{seed-1}) is exact if $\varphi \in \PO_1$ or if [$D_N^* = 0$ and $\varphi \in \PO_2$].
\end{lemma}

\begin{proof}
That the formula (\ref{seed-1}) is exact on $\PO_1$ is a triviality, although this fact will also be deduced, along with the rest of the conclusions, from the reasoning that follows.

\par We will write $E_n^*$ instead of $E_n^*(\varphi)$. Recall that, under (\ref{interlacing}), conditions (\ref{mesh-order-is-1}) and (\ref{points-order-is-1}) are equivalent by Lemma \ref{lemma-mesh-points-order-is-1}.

\par Take $\varphi \in {\rm C}^3([-1,1])$ and $n \in \{1,\dots,N\}$, and consider the Taylor expansions
\begin{multline}
\label{Tay-1} \varphi(\mu_{n + 1/2}) = \varphi(\mu_n) + \hsnp \varphi'(\mu_n) + \frac{(\hsnp)^2}{2} \varphi''(\mu_n)\\
+ \frac{(\hsnp)^3}{6} \varphi'''(\xi_{n+}^*),
\end{multline}
\begin{multline}
\label{Tay-2}\varphi(\mu_{n - 1/2}) = \varphi(\mu_n) - \hsnm \varphi'(\mu_n) + \frac{(\hsnm)^2}{2} \varphi''(\mu_n)\\
- \frac{(\hsnm)^3}{6} \varphi'''(\xi_{n-}^*).
\end{multline}

\par Subtracting (\ref{Tay-1}) and (\ref{Tay-2}) and dividing the result by $\mu_{n + 1/2} - \mu_{n - 1/2} = \hsnm + \hsnp$, we have
\begin{multline}
-E_n^* = \frac{\varphi(\mu_{n+1/2}) - \varphi(\mu_{n-1/2})}{\mu_{n+1/2} - \mu_{n-1/2}} - \varphi'(\mu_n) = \frac{(\hsnp)^2 - (\hsnm)^2}{2(\hsnm + \hsnp)} \varphi''(\mu_n)\\
+ \frac{(\hsnp)^3 \varphi'''(\xi_{n+}^*) + (\hsnm)^3 \varphi'''(\xi_{n-}^*)}{6(\hsnm + \hsnp)},
\end{multline}
or, taking account of
\begin{equation}
\frac{(\hsnp)^2 - (\hsnm)^2}{2(\hsnm + \hsnp)} = \frac{\hsnp - \hsnm}{2} = \dst{n},
\end{equation}
\begin{equation}
\label{Enstar}
E_n^* = -\dst{n} \varphi''(\mu_n) - \frac{(\hsnp)^3 \varphi'''(\xi_{n+}^*) + (\hsnm)^3 \varphi'''(\xi_{n-}^*)}{6(\hsnm + \hsnp)}.
\end{equation}

\par Now, since $\hsnm$ and $\hsnp$ are positive due to (\ref{interlacing}) and
\begin{multline}
\frac{(\hsnp)^3 + (\hsnm)^3}{6(\hsnm + \hsnp)} = \frac{(\hsnp)^2 - \hsnp \hsnm + (\hsnm)^2}{6}\\
= \frac{(\hst{n} + \dst{n})^2 - (\hst{n} + \dst{n}) (\hst{n} - \dst{n}) + (\hst{n} - \dst{n})^2}{6}\\
= \frac{(\hst{n})^2 + 3 (\dst{n})^2}{6} \leq \frac{(M_N^*/2)^2 + 3 (D_N^*)^2}{6} = \frac{(M_N^*)^2 + 12 (D_N^*)^2}{24},
\end{multline}
we get from Equation (\ref{Enstar})  the following inequality:
\begin{equation}
\label{eq-mother-formula-1}
\max_{1 \leq n \leq N} |E_n^*| \leq D_N^* \|\varphi''\|_\infty + \frac{(M_N^*)^2 + 12 (D_N^*)^2}{24} \|\varphi'''\|_\infty.
\end{equation}

\par The `if part' is a consequence of (\ref{eq-mother-formula-1}), (\ref{points-order-is-1}), and (\ref{mid-point-1-hyp}). Equation (\ref{eq-mother-formula-1}) also implies that the formula (\ref{seed-1}) is exact if $\varphi \in \PO_1$ or if [$D_N^* = 0$ and $\varphi \in \PO_2$].

\par The `only if part' can be proved in two steps:
\begin{itemize}
\item[Step 1] If the hypothesis (\ref{mid-point-1-hyp}) does not hold, that is, if $D_N^* \neq O(N^{-2})$, then $\max_{1 \leq n \leq N} |E_n^*| \neq O(N^{-2})$ for certain $\varphi \in {\rm C}^3([-1,1])$. Indeed, if one takes $\varphi(\mu) = \mu^2$, then  $E_n^* = - 2 \dst{n}$ by (\ref{Enstar}), and hence $\max_{1 \leq n \leq N} |E_n^*| = 2 D_N^* \neq O(N^{-2})$.

\item[Step 2] If the hypothesis (\ref{mid-point-1-hyp}) holds but the hypothesis (\ref{mesh-order-is-1}) does not hold, then $\max_{1 \leq n \leq N} |E_n^*| \neq O(N^{-2})$ for certain $\varphi \in {\rm C}^3([-1,1])$. To see this, let us take $\varphi(\mu) = \mu^3$. Then, $E_n^* = - (\hst{n})^2 - 3 (\dst{n})^2 - 6 \mu_n \dst{n}$ by (\ref{Enstar}). Now we will prove that $\max_{1 \leq n \leq N} |E_n^*| \neq O(N^{-2})$. Notice that $(\hst{n})^2 + 3 (\dst{n})^2 - 6 |\mu_n \dst{n}| \leq |E_n^*|$, and hence, for $n = 1,\dots,N$,
\begin{multline}
(\hst{n})^2 \leq |E_n^*| + 6 |\mu_n \dst{n}| - 3 (\dst{n})^2 \leq |E_n^*| + 6 |\mu_n \dst{n}|\\
\leq \max_{1 \leq n \leq N} |E_n^*| + 6 D_N^*,
\end{multline}
from where
\begin{equation}
(M_N^*)^2 \leq 4 \max_{1 \leq n \leq N} |E_n^*| + 24 D_N^*.
\end{equation}
So, $M_N^*$ would be $O(N^{-1})$, i.e., the hypothesis (\ref{mesh-order-is-1}) would be satisfied, if $\max_{1 \leq n \leq N} |E_n^*|$ were $O(N^{-2})$. This ends the proof of Step 2.
\end{itemize}

\par The examples above are also useful to demonstrate that the order $2$ cannot be improved:
\begin{itemize}
\item If $D_N^* \neq O(N^{-q})$ for all $q > 2$, then the example given by $\varphi(\mu) = \mu^2$ shows that $\max_{1 \leq n \leq N} |E_n^*| = 2 D_N^*$ is of the same order than $D_N^*$, so less than or equal to $2$.

\item If $D_N^* = O(N^{-q})$ for some $q > 2$, then the example given by $\varphi(\mu) = \mu^3$ shows that $\max_{1 \leq n \leq N} |E_n^*| \geq ((M_N^*)^2 - 24 D_N^*)/4 \geq ((2/N)^2 - 24 D_N^*)/4 = N^{-2} - 6 D_N^*$, and so $\max_{1 \leq n \leq N} |E_n^*|$ is again at most of order $2$. The inequality $M_N^* \geq 2/N$ follows from $2 = \sum_{n = 1}^N (\mu_{n + 1/2} - \mu_{n - 1/2}) \leq N M_N^*$.
\end{itemize}

\par This ends the proof of Lemma \ref{mother-formula-1}.
\end{proof}

\par The following result is analogous to Lemma \ref{mother-formula-1}, but contains a finer expression of the error term that will be needed later.

\begin{lemma}
\label{mother-formula-2}
Assume that the hypothesis (\ref{interlacing}) holds.
\begin{enumerate}
\item[(A)] The approximation
\begin{equation}
\label{seed-2}
\varphi'(\mu_{n+1/2}) \approx \frac{\varphi(\mu_{n+1}) - \varphi(\mu_n)}{\mu_{n+1} - \mu_n},\ n = 1,\dots,N-1,
\end{equation}
converges with order $2$ if, and only if, the hypotheses (\ref{mesh-order-is-1}) and (\ref{mid-point-2-hyp}) are met. More precisely, if $E_n(\varphi)$ is defined by
\begin{equation}
\label{EnDEF}
E_n(\varphi) = \varphi'(\mu_{n+1/2}) - \frac{\varphi(\mu_{n+1}) - \varphi(\mu_n)}{\mu_{n+1} - \mu_n},\ n = 1,\dots,N-1,
\end{equation}
then
\begin{equation}
\max_{1 \leq n \leq N-1} |E_n(\varphi)| = O(N^{-2})\ \mbox{for all}\ \varphi \in {\rm C}^3([-1,1])
\end{equation}
if, and only if, the hypotheses (\ref{mesh-order-is-1}) and (\ref{mid-point-2-hyp}) are met.

\par The maximal possible order is $2$.

\par Moreover, the formula (\ref{seed-2}) is exact if $\varphi \in \PO_1$ or if [$D_N = 0$ and $\varphi \in \PO_2$].

\item[(B)] If $\varphi \in {\rm C}^5([-1,1])$, then, for $n = 1,\dots,N-1$,
\begin{multline}
\label{Enclass5}
E_n(\varphi) = -d_n \varphi''(\mu_{n+1/2}) - \frac{h_n^2 + 3 d_n^2}{6} \varphi'''(\mu_{n+1/2})\\
 - \frac{h_n^2 d_n + d_n^3}{6} \varphi^{4)}(\mu_{n+1/2}) - \frac{\hnpe{5} \varphi^{5)}(\xinp) + \hnme{5} \varphi^{5)}(\xinm)}{120(\hnm + \hnp)}.
\end{multline}
\end{enumerate}
\end{lemma}

\begin{proof}
The proof of (A) is like that of Lemma \ref{mother-formula-1}. Let us prove (B).

\par Subtracting the Taylor expansions
\begin{multline}
\label{Tay-1-b} \varphi(\mu_{n+1}) = \varphi(\mu_{n + 1/2}) + \hnp \varphi'(\mu_{n + 1/2}) + \frac{\hnpe{2}}{2} \varphi''(\mu_{n + 1/2})\\
+ \frac{\hnpe{3}}{6} \varphi'''(\mu_{n + 1/2}) + \frac{\hnpe{4}}{24} \varphi^{4)}(\mu_{n + 1/2}) + \frac{\hnpe{5}}{120} \varphi^{5)}(\xinp),
\end{multline}
\begin{multline}
\label{Tay-2-b} \varphi(\mu_n) = \varphi(\mu_{n + 1/2}) - \hnm \varphi'(\mu_{n + 1/2}) + \frac{\hnme{2}}{2} \varphi''(\mu_{n + 1/2})\\
- \frac{\hnme{3}}{6} \varphi'''(\mu_{n + 1/2}) + \frac{\hnme{4}}{24} \varphi^{4)}(\mu_{n + 1/2}) - \frac{\hnme{5}}{120} \varphi^{5)}(\xinm),
\end{multline}
and then dividing the result by $\mu_{n+1} - \mu_n = \hnm + \hnp$, one gets
\begin{multline}
\label{quasi-Enclass5} -E_n = \frac{\varphi(\mu_{n+1}) - \varphi(\mu_n)}{\mu_{n+1} - \mu_n} - \varphi'(\mu_{n + 1/2})\\
= \frac{\hnpe{2} - \hnme{2}}{2(\hnm + \hnp)} \varphi''(\mu_{n + 1/2}) + \frac{\hnpe{3} + \hnme{3}}{6(\hnm + \hnp)} \varphi'''(\mu_{n + 1/2})\\
+ \frac{\hnpe{4} - \hnme{4}}{24(\hnm + \hnp)} \varphi^{4)}(\mu_{n + 1/2}) + \frac{\hnpe{5} \varphi^{5)}(\xinp) + \hnme{5} \varphi^{5)}(\xinm)}{120(\hnm + \hnp)}.
\end{multline}

\par Finally, the error representation (\ref{Enclass5}) results from (\ref{quasi-Enclass5}) and the following equalities:
\begin{align}
\label{class5-a} \frac{\hnpe{2} - \hnme{2}}{2(\hnm + \hnp)} & = \frac{\hnp - \hnm}{2} = d_n,\\
\label{class5-b} \frac{\hnpe{3} + \hnme{3}}{6(\hnm + \hnp)} & = \frac{1}{6}(\hnpe{2} - \hnp \hnm + \hnme{2}) = \frac{h_n^2 + 3 d_n^2}{6},\\
\label{class5-c} \frac{\hnpe{4} - \hnme{4}}{24(\hnm + \hnp)} & = \frac{(\hnpe{2} + \hnme{2}) (\hnp - \hnm)}{24} = \frac{h_n^2 d_n + d_n^3}{6},
\end{align}
where we have used the identities $\hnm = h_n - d_n$ and $\hnp = h_n + d_n$.
\end{proof}

\section{Some general comments}\label{SECTION-general-comments}
We will describe in the following sections difference schemes for approximating the FP angular diffusion operator $\Delta_{\rm FP} f$ defined by Equation (\ref{AICSO}). In what follows, $\Delta_{{\rm FP},N} f(\mu_n)$ will stand for an approximation of $\Delta_{\rm FP} f(\mu_n)$ obtained on a mesh of $N$ nodes.

\begin{defi}
For each $n = 1,\dots,N$, we define the truncation error $R_n(f)$ as
\begin{equation}
R_n(f) = \Delta_{\rm FP} f(\mu_n) - \Delta_{{\rm FP},N} f(\mu_n).
\end{equation}
\end{defi}

\begin{defi}
A numerical scheme for computing $\Delta_{{\rm FP},N} f(\mu_n)$
\begin{enumerate}
\item Converges for the function $f$ if
\begin{equation}
\lim_{N \to \infty} \max_{1 \leq n \leq N} |R_n(f)| = 0.
\end{equation}

\item Converges with (at least) order $p$ for the function $f$ if
\begin{equation}
\max_{1 \leq n \leq N} |R_n(f)| = O(N^{-p})
\end{equation}
for certain positive real number $p$.

\item Converges with order $p$ if converges with order $p$ for all $f$ regular enough, which in this paper will mean that there exists $k \in \NA$ such that converges with order $p$ for all $f \in {\rm C}^k([-1,1])$.
\end{enumerate}
\end{defi}

\par As was anticipated in the introduction, a particular case of DOM schemes will have a special relevance in this paper: the GL schemes, the definition of which is formalized as follows.

\begin{defi}
Any scheme that takes as $\{\mu_n\}_{n=1}^N$ the set of GL nodes will be called a GL scheme.
\end{defi}

\par The FPE is frequently solved with a GL scheme, in which, sometimes, the set $\{\mu_{n+1/2}\}_{n = 0}^N$ is constructed from the GL weights. Whenever a GL scheme is used, it is usually considered that there are two possible modes of application (see for instance \cite{GALP23}):
\begin{enumerate}
\item Full range (FR) mode: nodes and weights are those of the GL formula of $N$ points in $(-1,1)$. Automatically, this refines the mesh in the vicinity of $-1$ and $1$. The FPE degenerates at $\mu = 0$, and so the node $0$ is typically avoided by taking $N$ even, but the parity of $N$ is not at all relevant when studying the convergence of the schemes that discretize the angular diffusion operator in isolation. We think that the ideas contained in this paper can be used to design a DOM scheme for the FPE which can use $N$ odd while maintaining good properties as order of convergence and discrete moments preservation, but this will be part of future research.

\item Half range (HR) mode: nodes and weights are those of the GL formula of $N$ points in $(-1,0)$ and those of the GL formula of $N$ points in $(0,1)$. In this way, one has a total amount of $2N$ nodes. Automatically, this avoids the node $0$ and refines the mesh in the vicinity of $-1$, $0$, and $1$. It is clear that an equivalent explanation can be given with $N$ nodes as long as $N$ is even, but we will always consider $2N$ nodes when operating GL schemes in HR mode.
\end{enumerate}

\subsection{The zeroth and first moment properties}
Associated with the FP Laplacian, there are two properties of interest, namely the zeroth and the first moment properties:
\begin{align}
\label{zeroth-moment} & \int_{-1}^1 \Delta_{\rm FP} f(\mu)\ d\mu = 0,\\
\label{first-moment} & \int_{-1}^1 \mu \Delta_{\rm FP} f(\mu)\ d\mu = -2 \int_{-1}^1 \mu f(\mu)\ d\mu,
\end{align}
both of which are easy to verify. The reader can think about how these properties should be written for diffusivities other than $\dd(\mu) = 1 - \mu^2$.

\par According to \cite{MO85}, it is of interest that the schemes satisfy discrete versions of these two properties.

\begin{defi}\label{DEF-discrete-moment}
We say that a GL scheme
\begin{itemize}
\item Satisfies the discrete zeroth moment property (or preserves the zeroth moment) if
\begin{equation}
\sum_{n = 1}^N w_n \Delta_{{\rm FP},N} f(\mu_n) = 0.
\end{equation}

\item Satisfies the discrete first moment property (or preserves the first moment) if
\begin{equation}
\sum_{n = 1}^N w_n \mu_n \Delta_{{\rm FP},N} f(\mu_n) = -2 \sum_{n = 1}^N w_n \mu_n f(\mu_n),
\end{equation}
\end{itemize}
where $\{w_n\}_{n=1}^N$ are the GL weights.
\end{defi}

\par Obviously, Definition \ref{DEF-discrete-moment} relies on GL quadrature, which is natural for GL schemes, but, when dealing with a non\=/GL scheme, an analogous definition can be written based on some other appropriate quadrature rule.

\section{Schemes of type I}\label{SECTION-type-I}
After (\ref{i-classical-deduction}), and noticing that $\dd(\mu_{1/2}) = \dd(\mu_{N+1/2}) = 0$, let us consider the following scheme:

\begin{equation}
\label{Type-I-scheme-a}
\Delta_{{\rm FP},N} f(\mu_1) = \frac{\dd(\mu_{1+1/2}) \frac{f(\mu_2) - f(\mu_1)}{\mu_2 - \mu_1}}{\mu_{1+1/2} + 1},
\end{equation}

\begin{eqnarray}
\nonumber & & \Delta_{{\rm FP},N} f(\mu_n) = \frac{\dd(\mu_{n+1/2}) \frac{f(\mu_{n+1}) - f(\mu_n)}{\mu_{n+1} - \mu_n} - \dd(\mu_{n-1/2}) \frac{f(\mu_n) - f(\mu_{n-1})}{\mu_n - \mu_{n-1}}}{\mu_{n+1/2} - \mu_{n-1/2}}\\
\label{Type-I-scheme-b} & & \mbox{for}\ n = 2,\dots,N-1,
\end{eqnarray}

\begin{equation}
\label{Type-I-scheme-c}
\Delta_{{\rm FP},N} f(\mu_N) = \frac{- \dd(\mu_{N-1/2}) \frac{f(\mu_N) - f(\mu_{N-1})}{\mu_N - \mu_{N-1}}}{1 - \mu_{N-1/2}}.
\end{equation}

\par The scheme (\ref{Type-I-scheme-a})--(\ref{Type-I-scheme-c}) can be written simply as
\begin{eqnarray}
\nonumber & & \Delta_{{\rm FP},N} f(\mu_n) = \frac{\dd(\mu_{n+1/2}) \frac{f(\mu_{n+1}) - f(\mu_n)}{\mu_{n+1} - \mu_n} - \dd(\mu_{n-1/2}) \frac{f(\mu_n) - f(\mu_{n-1})}{\mu_n - \mu_{n-1}}}{\mu_{n+1/2} - \mu_{n-1/2}}\\
\label{Type-I-scheme} & & \mbox{for}\ n = 1,\dots,N,
\end{eqnarray}
understanding that the terms containing the undefined nodes $\mu_0$ and $\mu_{N+1}$ must be ignored as they are multiplied by zero.

\par This is really a family of schemes depending upon the choice of the nodes $\mu_n$ and the points $\mu_{n+1/2}$. We shall refer to the members of this family as schemes of type I.

\subsection{First example: Lee's scheme}\label{subsec-Lee}
If $\{\mu_n\}_{n=1}^N$ are the GL nodes, and the points $\{\mu_{n+1/2}\}_{n = 0}^N$ are defined by $\mu_{1/2} = -1$, $\mu_{n+1/2} = (\mu_n + \mu_{n+1})/2$ for $n = 1,\dots,N-1$, $\mu_{N+1/2} = 1$, one recovers the scheme used, in chronological order, by Lee in \cite{LE62}, Antal and Lee in \cite{ANLE76} and Mehlhorn and Duderstadt in \cite {MEDU80}. According to the literature, this scheme was the standard in the nuclear engineering community from the sixties of the past century until the appearance of the Haldy\=/Ligou's scheme, which in turn was soon substituted by the Morel's scheme, to be described later.

\subsection{Second example: Haldy\=/Ligou's scheme}\label{subsec-HL}
If $\{\mu_n\}_{n=1}^N$ are the GL nodes, and the points $\{\mu_{n+1/2}\}_{n = 0}^N$ are defined by $\mu_{1/2} = -1$, $\mu_{n+1/2} = \mu_{n - 1/2} + w_n$ for $n = 1,\dots,N$, being $\{w_n\}_{n=1}^N$ the GL weights, one recovers the scheme used by Haldy and Ligou in \cite{HALI80}.

\par Since
\begin{equation}
\sum_{n=1}^N w_n = 2,
\end{equation}
points $\mu_{n+1/2}$ are antisymmetric with respect to $0$:
\begin{equation}\label{HLpoints-antisym}
\mu_{n+1/2} = -\mu_{N-n+1/2}\ \mbox{for}\ n = 0,\dots,N.
\end{equation}

\par  In particular, one always obtains $\mu_{N+1/2} = 1$ and, if $N$ is even, $\mu_{(N/2)+1/2} = \hmu_{(N/2)+1/2} = 0$.

\par When programming this scheme, it is convenient to take advantage of Equation (\ref{HLpoints-antisym}) by calculating only those points $\mu_{n+1/2}$ that belong to $[-1,0]$, and then determining the ones in $(0,1]$ by means of the antisymmetry. In this way, roundoff errors are reduced.

\par So, in FR mode, this scheme reads as follows:
\begin{eqnarray}
\nonumber & & \Delta_{{\rm FP},N} f(\mu_n) = \frac{\dd(\mu_{n+1/2}) \frac{f(\mu_{n+1}) - f(\mu_n)}{\mu_{n+1} - \mu_n} - \dd(\mu_{n-1/2}) \frac{f(\mu_n) - f(\mu_{n-1})}{\mu_n - \mu_{n-1}}}{w_n}\\
\label{Haldy-Ligou-scheme} & & \mbox{for}\ n = 1,\dots,N.
\end{eqnarray}

\par It can be seen as an evolution of Lee's scheme designed so that the discrete zeroth moment property is satisfied while remaining a GL scheme.

\par When used in HR mode, Haldy\=/Ligou's scheme still satisfies the discrete zeroth moment property, but this is unimportant, since it is no longer convergent.

\subsection{Analysis of convergence}
Results in this subsection hold for generic diffusivities and not only for $\dd(\mu) = 1 - \mu^2$.

\par We start with a result on the error representation.
\begin{prop}[error representation for schemes of type I]
\label{ER-type-I} Let $\dd$ be a function of class ${\rm C}^1([-1,1])$ such that $\dd(-1) = \dd(1) = 0$. Suppose that $f \in {\rm C}^2([-1,1])$ and that $\Delta_{{\rm FP},N} f(\mu_n)$ is defined by Equation (\ref{Type-I-scheme}). Then, for $n = 1,\dots,N$,
\begin{equation}\label{DEFRn}
\Delta_{\rm FP} f(\mu_n) = \Delta_{{\rm FP},N} f(\mu_n) + R_n(f),
\end{equation}
with
\begin{equation}
\label{RnDEF}
R_n(f) = \varepsilon_n(f) + E_n^\star(\dd f'),
\end{equation}
being
\begin{align}
\label{eps1DEF} \varepsilon_1(f) & = \frac{\dd(\mu_{1+1/2})E_1(f)}{\mu_{1+1/2} + 1},\\
\label{epsnDEF} \varepsilon_n(f) & = \frac{\dd(\mu_{n+1/2})E_n(f) - \dd(\mu_{n-1/2})E_{n-1}(f)}{\mu_{n+1/2} - \mu_{n-1/2}}\ \mbox{for}\ n = 2,\dots,N-1,\\
\label{epsNDEF} \varepsilon_N(f) & = -\frac{\dd(\mu_{N-1/2})E_{N-1}(f)}{1 - \mu_{N-1/2}}.
\end{align}
\par In the expressions above, $E_n^\star(\dd f')$ and $E_n(f)$ are those defined by Equations (\ref{EnstarDEF}) and (\ref{EnDEF}), respectively.
\end{prop}

\begin{proof}
$\Delta_{\rm FP} f$ is well defined in the classical sense because $\dd \in {\rm C}^1([-1,1])$ and $f \in {\rm C}^2([-1,1])$.

\par Using (\ref{EnstarDEF}), $\dd(\mu_{1/2}) = 0$, and (\ref{EnDEF}),
\begin{multline}
\Delta_{\rm FP} f(\mu_1) = \frac{\dd(\mu_{1+1/2}) f'(\mu_{1+1/2}) - \dd(\mu_{1/2}) f'(\mu_{1/2})}{\mu_{1+1/2} + 1} + E_1^\star(\dd f')\\
= \frac{\dd(\mu_{1+1/2}) f'(\mu_{1+1/2})}{\mu_{1+1/2} + 1} + E_1^\star(\dd f')\\
= \frac{\dd(\mu_{1+1/2})}{\mu_{1+1/2} + 1} \left\{\frac{f(\mu_2) - f(\mu_1)}{\mu_2 - \mu_1} + E_1(f)\right\} + E_1^\star(\dd f')\\
= \Delta_{{\rm FP},N} f(\mu_1) + \varepsilon_1(f) + E_1^\star(\dd f'),
\end{multline}
with $\varepsilon_1(f)$ given by (\ref{eps1DEF}). The missing proofs can be done analogously.
\end{proof}

\par Our goal is to fix certain conditions on the set of nodes and points so that the scheme converges with order $2$. Thanks to Equation (\ref{RnDEF}) and Lemma \ref{mother-formula-1}, the point is to establish conditions for $\max_{1 \leq n \leq N} |\varepsilon_n(f)|$ to be a  $O(N^{-2})$ when $f$ is regular enough.

\par As anticipated by (\ref{eps1DEF}) and (\ref{epsNDEF}), the determination of bounds for $|\varepsilon_1(f)|$ and $|\varepsilon_N(f)|$ is special because $\dd(\mu_{1/2}) = \dd(\mu_{N+1/2}) = 0$. It turns out to be a very easy task.

\begin{prop}[bound for $\max \{|\varepsilon_1(f)|, |\varepsilon_N(f)|\}$]
\label{BOUND-for-eps1-and-epsN}
Suppose that $\dd$ is a function of class ${\rm C}^1([-1,1])$ such that $\dd(-1) = \dd(1) = 0$. Let $f$ be a function of class ${\rm C}^2([-1,1])$ and let $\varepsilon_1(f)$, $\varepsilon_N(f)$ be the quantities defined by Equations (\ref{eps1DEF}) and (\ref{epsNDEF}), respectively. Then,
\begin{equation}
\max \{|\varepsilon_1(f)|, |\varepsilon_N(f)|\} \leq \|\dd'\|_\infty \left( \max_{1 \leq n \leq N-1} |E_n(f)| \right).
\end{equation}
\end{prop}

\begin{proof}
Notice that $\mu_{1+1/2} + 1 = 2 h_1^*$. Then, Equation (\ref{eps1DEF}) and the equality
\begin{equation}
\dd(\mu_{1+1/2}) = \dd(\mu_{1/2}) + 2 h_1^* \dd'(c_1) = 2 h_1^* \dd'(c_1),
\end{equation}
obtained by means of Taylor's theorem, imply
\begin{equation}
|\varepsilon_1(f)| \leq \|\dd'\|_\infty \left( \max_{1 \leq n \leq N-1} |E_n(f)| \right).
\end{equation}

\par Proceeding in a similar way, one sees that the same upper bound is valid for $|\varepsilon_N(f)|$.
\end{proof}

\par Obtaining an appropriate bound for $\max_{2 \leq n \leq N-1} |\varepsilon_n(f)|$ is much more difficult. We need to introduce some new definitions and, as will be seen in the proof of Proposition \ref{BOUND-for-epsn} below, use the second part of Lemma \ref{mother-formula-2} and break the problem into several simpler ones.

\begin{defi}\label{ANDEF}
For $n = 1,\dots,N-1$, $a_n = d_n + d_n^*$, and
\begin{equation}
A_N = \max_{1 \leq n \leq N-1} |a_n|.
\end{equation}
\end{defi}

\begin{defi}\label{BNDEF}
For $n = 2,\dots,N$, $b_n = d_{n-1} + d_n^*$, and
\begin{equation}
B_N = \max_{2 \leq n \leq N} |b_n|.
\end{equation}
\end{defi}

\begin{defi}\label{CNDEF}
$C_N = D_N + D_N^*$.
\end{defi}

\par Notice that $C_N$ can be used to bound both $A_N$ and $B_N$.

\begin{defi}\label{betaNDEF}
$\beta_N(\dd)$ is the number defined by
\begin{equation}
\beta_N(\dd) = \max_{2 \leq n \leq N-1} \left|\frac{(d_{n-1} - d_n) \dd(\mu_{n+1/2})}{\mu_{n + 1/2} - \mu_{n - 1/2}}\right|
\end{equation}
or, equivalently,
\begin{equation}
\beta_N(\dd) = \max_{2 \leq n \leq N-1} \left|\frac{(d_{n-1} - d_n) \dd(\mu_{n+1/2})}{2 h_n^\star}\right|.
\end{equation}
\end{defi}

\begin{prop}[bound for $\max_{2 \leq n \leq N-1} |\varepsilon_n(f)|$]
\label{BOUND-for-epsn}
Suppose that $\dd$ is a function of class ${\rm C}^1([-1,1])$. Let us understand that $\beta_N = \beta_N(\dd)$ and let $f$ be a function of class ${\rm C}^5([-1,1])$. Fix $n \in \{2,\dots,N-1\}$ and let $\varepsilon_n(f)$ be the quantity defined by Equation (\ref{epsnDEF}).
Then,
\begin{equation}
\label{epsnclass5}
\varepsilon_n(f) = \varepsilon_n^{(1)}(f) + \varepsilon_n^{(2)}(f) + \varepsilon_n^{(3)}(f) + \varepsilon_n^{(4)}(f),
\end{equation}
with
\begin{align}
\label{boundeps1n} |\varepsilon_n^{(1)}(f)| & \leq \beta_N \|f''\|_\infty + D_N \|(\dd f'')'\|_\infty,\\
\nonumber |\varepsilon_n^{(2)}(f)| & \leq \frac{\beta_N (3 D_N + C_N)}{3} \|f'''\|_\infty + \frac{C_N}{3} \|\dd f'''\|_\infty\\
\label{boundeps2n} & \hspace{30mm} + \frac{M_N^2 + 12 D_N^2}{24} \|(\dd f''')'\|_\infty,\\
\nonumber |\varepsilon_n^{(3)}(f)| & \leq \frac{\beta_N \{(M_N^*)^2 + 8 C_N D_N + 4 C_N^2 + 12 D_N^2\}}{24} \|f^{4)}\|_\infty\\
\label{boundeps3n} & \hspace{5mm} + \frac{C_N D_N}{3} \|\dd f^{4)}\|_\infty + \frac{M_N^2 D_N + 4 D_N^3}{24} \|(\dd f^{4)})'\|_\infty,\\
\label{boundeps4n} |\varepsilon_n^{(4)}(f)| & \leq \frac{Z_N}{960} \|\dd\|_\infty \|f^{5)}\|_\infty,
\end{align}
where
\begin{multline}\label{DEFZN}
Z_N = (M_N^*)^3 + 8 (M_N^*)^2 C_N + 24 M_N^* C_N^2 + 40 M_N^* D_N^2\\
+ 32 C_N^3 + 160 C_N D_N^2 + \frac{16}{m_N^\star} (C_N^4 + 10 C_N^2 D_N^2 + 5 D_N^4).
\end{multline}
\end{prop}

\begin{proof}
According to Equations (\ref{Enclass5}), in the second part of Lemma \ref{mother-formula-2}, and (\ref{epsnDEF}),
\begin{multline}
\varepsilon_n(f) = \frac{1}{2 h_n^*} \left\{\dd(\mu_{n-1/2}) \left[d_{n-1} f''(\mu_{n-1/2}) \phantom{\frac{h_{n-1}^2}{6}} \right. \right.\\
+ \frac{h_{n-1}^2 + 3 d_{n-1}^2}{6} f'''(\mu_{n-1/2}) + \frac{h_{n-1}^2 d_{n-1} + d_{n-1}^3}{6} f^{4)}(\mu_{n-1/2})\\
+ \left. \frac{h_{(n-1)+}^5 f^{5)}(\xi_{(n-1)+}) + h_{(n-1)-}^5 f^{5)}(\xi_{(n-1)-})}{120 (h_{(n-1)-} + h_{(n-1)+})} \right]\\
- \dd(\mu_{n+1/2}) \left[d_n f''(\mu_{n+1/2})+ \frac{h_n^2 + 3 d_n^2}{6} f'''(\mu_{n+1/2})  \right.\\
+ \left. \left. \frac{h_n^2 d_n + d_n^3}{6} f^{4)}(\mu_{n+1/2}) + \frac{h_{n+}^5 f^{5)}(\xi_{n+}) + h_{n-}^5 f^{5)}(\xi_{n-})}{120 (h_{n-} + h_{n+})} \right] \right\},
\end{multline}
which gives (\ref{epsnclass5}) with
\begin{multline}
\label{eq-eps1n}
\varepsilon_n^{(1)}(f) = \frac{1}{2 h_n^*} \{ d_{n-1} \dd(\mu_{n-1/2}) f''(\mu_{n-1/2})\\
-d_n \dd(\mu_{n+1/2}) f''(\mu_{n+1/2}) \},
\end{multline}
\begin{multline}
\label{eq-eps2n}
\varepsilon_n^{(2)}(f) = \frac{1}{2 h_n^*} \left\{ \frac{h_{n-1}^2 + 3 d_{n-1}^2}{6} \dd(\mu_{n-1/2}) f'''(\mu_{n-1/2}) \right.\\
\left. - \frac{h_n^2 + 3 d_n^2}{6} \dd(\mu_{n+1/2}) f'''(\mu_{n+1/2}) \right\},
\end{multline}
\begin{multline}
\label{eq-eps3n}
\varepsilon_n^{(3)}(f) = \frac{1}{2 h_n^*} \left\{ \frac{h_{n-1}^2 d_{n-1} + d_{n-1}^3}{6} \dd(\mu_{n-1/2}) f^{4)}(\mu_{n-1/2}) \right.\\
\left. - \frac{h_n^2 d_n + d_n^3}{6} \dd(\mu_{n+1/2}) f^{4)}(\mu_{n+1/2}) \right\},
\end{multline}
\begin{multline}
\label{eq-eps4n}
\varepsilon_n^{(4)}(f)\\
= \frac{1}{2 h_n^*} \left\{ \dd(\mu_{n-1/2}) \frac{h_{(n-1)+}^5 f^{5)}(\xi_{(n-1)+}) + h_{(n-1)-}^5 f^{5)}(\xi_{(n-1)-})}{120 (h_{(n-1)-} + h_{(n-1)+})} \right.\\
\left. - \dd(\mu_{n+1/2}) \frac{h_{n+}^5 f^{5)}(\xi_{n+}) + h_{n-}^5 f^{5)}(\xi_{n-})}{120 (h_{n-} + h_{n+})} \right\}.
\end{multline}

\par We will prove (\ref{boundeps1n})--(\ref{boundeps3n}) firstly and leave the proof of (\ref{boundeps4n}) for later.

\par Notice that, for $r = 1,2,3$, we have by Taylor that
\begin{multline}\label{Taylor}
\dd(\mu_{n-1/2}) f^{r+1)}(\mu_{n-1/2}) = \dd(\mu_{n+1/2}) f^{r+1)}(\mu_{n+1/2})\\
- 2 h_n^* (\dd f^{r+1)})'(c^{(r)}_n),
\end{multline}
and hence the expressions (\ref{eq-eps1n})--(\ref{eq-eps3n}) can be rewritten as follows:
\begin{equation}
\label{eq-eps1n-BIS}
\varepsilon_n^{(1)}(f) = \frac{(d_{n-1} - d_n) \dd(\mu_{n+1/2})}{2 h_n^*} f''(\mu_{n+1/2}) - d_{n-1} (\dd f'')'(c^{(1)}_n),
\end{equation}
\begin{multline}
\label{eq-eps2n-BIS}
\varepsilon_n^{(2)}(f) = \frac{(h_{n-1}^2 + 3 d_{n-1}^2) - (h_n^2 + 3 d_n^2)}{12 h_n^*} \dd(\mu_{n+1/2}) f'''(\mu_{n+1/2})\\
- \frac{(h_{n-1}^2 + 3 d_{n-1}^2)}{6} (\dd f''')'(c^{(2)}_n),
\end{multline}
\begin{multline}
\label{eq-eps3n-BIS}
\varepsilon_n^{(3)}(f) = \frac{(h_{n-1}^2 d_{n-1} + d_{n-1}^3) - (h_n^2 d_n + d_n^3)}{12 h_n^*} \dd(\mu_{n+1/2}) f^{4)}(\mu_{n+1/2})\\
- \frac{(h_{n-1}^2 d_{n-1} + d_{n-1}^3)}{6} (\dd f^{4)})'(c^{(3)}_n).
\end{multline}

\par Now we proceed to bound each of these three terms separately.

\begin{itemize}
\item Bound for $|\varepsilon_n^{(1)}(f)|$: the bound (\ref{boundeps1n}) follows immediately from (\ref{eq-eps1n-BIS}).

\item Bound for $|\varepsilon_n^{(2)}(f)|$: thanks to Equation (\ref{odd}) we have
\begin{multline}
h_{n-1}^2 - h_n^2 = (h_{n-1} + h_n) (h_{n-1} - h_n)\\
= \{2 h_n^* - (d_{n-1} - d_n)\} (h_{n-1} - h_n).
\end{multline}

\par Thus, noticing that $h_{n-1} - h_n = -a_n - b_n$, which holds in virtue of Definitions \ref{ANDEF}, \ref{BNDEF} and Equations (\ref{ds-and-hs}), (\ref{ds-and-hs-BIS}),\footnote{The identity $h_n - h_{n-1} = a_n + b_n$ is also a rewriting of Equation (\ref{eq-lemma-mean-value}) in Lemma \ref{lemma-mean-value}.} we arrive at
\begin{multline}
\frac{(h_{n-1}^2 + 3 d_{n-1}^2) - (h_n^2 + 3 d_n^2)}{12 h_n^*}\\
= \frac{\{2 h_n^* - (d_{n-1} - d_n)\} (-a_n - b_n) + 3 (d_{n-1} + d_n) (d_{n-1} - d_n)}{12 h_n^*}\\
= -\frac{a_n + b_n}{6}\\
+ \frac{\{3 (d_{n-1} + d_n)+ a_n + b_n\} (d_{n-1} - d_n)}{12 h_n^*},
\end{multline}
and then Equation (\ref{eq-eps2n-BIS}) can be rewritten as
\begin{multline}
\label{eq-eps2n-BIS-BIS}
\varepsilon_n^{(2)}(f) = -\frac{a_n + b_n}{6} \dd(\mu_{n+1/2}) f'''(\mu_{n+1/2})\\
+ \frac{\{3 (d_{n-1} + d_n) + a_n + b_n\} }{6} \frac{(d_{n-1} - d_n) \dd(\mu_{n+1/2})}{2 h_n^*} f'''(\mu_{n+1/2})\\
- \frac{(h_{n-1}^2 + 3 d_{n-1}^2)}{6} (\dd f''')'(c^{(2)}_n),
\end{multline}
which implies
\begin{multline}
\label{quasi-boundeps2n}
|\varepsilon_n^{(2)}(f)| \leq \frac{C_N}{3} \|\dd f'''\|_\infty + \frac{3 D_N + C_N}{3} \beta_N \|f'''\|_\infty\\
+ \frac{1}{6} \{(M_N/2)^2 + 3 D_N^2\} \|(\dd f''')'\|_\infty.
\end{multline}

\par Finally, observe that (\ref{quasi-boundeps2n}) is equivalent to (\ref{boundeps2n}).

\item Bound for $|\varepsilon_n^{(3)}(f)|$: keeping Equation (\ref{eq-eps3n-BIS}) in mind, we will begin by obtaining expressions for $h_{n-1}^2 d_{n-1} - h_n^2 d_n$ and for $d_{n-1}^3 - d_n^3$ that allow us to bound $|\varepsilon_n^{(3)}(f)|$ in an optimal way. Taking into account the previous bounds, we realize that it is convenient to bring up the $d_{n-1} - d_n$ factor as many times as possible.

\par The easiest part is $d_{n-1}^3 - d_n^3$:
\begin{equation}
\label{eq-1-bound-eps3n}
d_{n-1}^3 - d_n^3 = (d_{n-1}^2 + d_{n-1} d_n + d_n^2) (d_{n-1} - d_n).
\end{equation}

\par Let us proceed now with $h_{n-1}^2 d_{n-1} - h_n^2 d_n$. It is known, from Equations (\ref{ds-and-hs}) and (\ref{ds-and-hs-BIS}), that $h_n = h_n^* + a_n$ and $h_{n-1} = h_n^* - b_n$. Hence,
\begin{multline}
\label{eq-2-bound-eps3n} h_{n-1}^2 d_{n-1} - h_n^2 d_n = (h_n^* - b_n)^2 d_{n-1} - (h_n^* + a_n)^2 d_n\\
= (h_n^*)^2 (d_{n-1} - d_n) - 2 h_n^* (b_n d_{n-1} + a_n d_n)\\
+ b_n^2 d_{n-1} - a_n^2 d_n.
\end{multline}

\par Next step is to prove that $b_n^2 d_{n-1} - a_n^2 d_n$ is a multiple of $d_{n-1} - d_n$. Note that, in virtue of Definitions \ref{ANDEF} and \ref{BNDEF}, $b_n - a_n = d_{n-1} - d_n$. So,
\begin{multline}
b_n^2 d_{n-1} - a_n^2 d_n = b_n^2 (d_{n-1} - d_n) + (b_n^2 - a_n^2) d_n\\
= b_n^2 (d_{n-1} - d_n) + (b_n + a_n) (b_n - a_n) d_n\\
= \{ b_n^2 + (b_n + a_n) d_n \} (d_{n-1} - d_n),
\end{multline}
and Equation (\ref{eq-2-bound-eps3n}) becomes
\begin{multline}
\label{eq-3-bound-eps3n} h_{n-1}^2 d_{n-1} - h_n^2 d_n\\
= \{(h_n^*)^2 + b_n^2 + (a_n + b_n) d_n\} (d_{n-1} - d_n)\\
- 2 h_n^* (b_n d_{n-1} + a_n d_n).
\end{multline}

\par In summary,
\begin{multline}
(h_{n-1}^2 d_{n-1} + d_{n-1}^3) - (h_n^2 d_n + d_n^3)\\
= \{(h_n^*)^2 + b_n^2 + (a_n + b_n) d_n + d_{n-1}^2 + d_{n-1} d_n + d_n^2\} (d_{n-1} - d_n)\\
- 2 h_n^* (b_n d_{n-1} + a_n d_n).
\end{multline}

\par If we define now
\begin{equation}
x_n = (h_n^*)^2 + b_n^2 + (a_n + b_n) d_n + d_{n-1}^2 + d_{n-1} d_n + d_n^2,
\end{equation}
Equation (\ref{eq-eps3n-BIS}) can be rewritten as follows:
\begin{multline}
\label{eq-eps3n-BIS-BIS}
\varepsilon_n^{(3)}(f) = \frac{x_n}{6} \frac{(d_{n-1} - d_n) \dd(\mu_{n+1/2})}{2 h_n^*}  f^{4)}(\mu_{n+1/2})\\
- \frac{b_n d_{n-1} + a_n d_n}{6} \dd(\mu_{n+1/2}) f^{4)}(\mu_{n+1/2})\\
- \frac{(h_{n-1}^2 d_{n-1} + d_{n-1}^3)}{6} (\dd f^{4)})'(c^{(3)}_n).
\end{multline}

\par Finally, taking account of
\begin{align}
|x_n| & \leq (M_N^*/2)^2 + C_N^2 + 2 C_N D_N + 3 D_N^2,\\
|b_n d_{n-1} + a_n d_n| & \leq 2 C_N D_N,
\end{align}
the bound (\ref{boundeps3n}) is deduced from Equation (\ref{eq-eps3n-BIS-BIS}).
\end{itemize}

\par We now proceed with the proof of (\ref{boundeps4n}).

\par Firstly note that, for $k \in \{n-1,n\}$,
\begin{multline}
\frac{h_{k+}^5 + h_{k-}^5}{h_{k-} + h_{k+}} = h_{k+}^4 - h_{k+}^3 h_{k-} + h_{k+}^2 h_{k-}^2 - h_{k+} h_{k-}^3 + h_{k-}^4\\
= h_k^4 + 10 h_k^2 d_k^2 + 5 d_k^4.
\end{multline}
\par The equalities $h_{k-} = h_k - d_k$ and $h_{k+} = h_k + d_k$ have been used in the last step.

\par Then, in virtue of Equation (\ref{eq-eps4n}) and the positivity of $h_{k-}$ and $h_{k+}$,
\begin{multline}
\label{coarse-bound}
|\varepsilon_n^{(4)}(f)| \leq \frac{\|\dd\|_\infty \|f^{5)}\|_\infty}{240} \frac{1}{h_n^*} (h_{n-1}^4 + 10 h_{n-1}^2 d_{n-1}^2 + 5 d_{n-1}^4\\
+ h_n^4 + 10 h_n^2 d_n^2 + 5 d_n^4).
\end{multline}

\par This bound can be improved proceeding as follows. Since $h_{n-1} = h_n^* - b_n$ and $h_n = h_n^* + a_n$, we have
\begin{multline}
h_{n-1}^4 + 10 h_{n-1}^2 d_{n-1}^2 = (h_n^*)^4 - 4 (h_n^*)^3 b_n + 6 (h_n^*)^2 b_n^2 + 10 (h_n^*)^2 d_{n-1}^2\\
- 4 h_n^* b_n^3 - 20 h_n^* b_n d_{n-1}^2 + b_n^4 + 10 b_n^2 d_{n-1}^2
\end{multline}
and
\begin{multline}
h_n^4 + 10 h_n^2 d_n^2 = (h_n^*)^4 + 4 (h_n^*)^3 a_n + 6 (h_n^*)^2 a_n^2 + 10 (h_n^*)^2 d_n^2\\
+ 4 h_n^* a_n^3 + 20 h_n^* a_n d_n^2 + a_n^4 + 10 a_n^2 d_n^2,
\end{multline}
which allows rewriting the inequality (\ref{coarse-bound}) as
\begin{multline}
|\varepsilon_n^{(4)}(f)| \leq \frac{\|\dd\|_\infty \|f^{5)}\|_\infty}{240} \{ (h_n^*)^3 - 4 (h_n^*)^2 b_n + 6 h_n^* b_n^2 + 10 h_n^* d_{n-1}^2 - 4 b_n^3\\
- 20 b_n d_{n-1}^2 + (h_n^*)^3 + 4 (h_n^*)^2 a_n + 6 h_n^* a_n^2 + 10 h_n^* d_n^2 + 4 a_n^3 + 20 a_n d_n^2\\
+ \frac{1}{h_n^*} (b_n^4 + 10 b_n^2 d_{n-1}^2 + 5 d_{n-1}^4 + a_n^4 + 10 a_n^2 d_n^2 + 5 d_n^4) \}.
\end{multline}

\par The proof ends by using the bounds $h_n^* \leq M_N^*/2$, $|b_n| \leq C_N$, $|a_n| \leq C_N$, $|d_n| \leq D_N$, and $1/h_n^* \leq 2/m_N^*$.
\end{proof}

\par We can now state the following result of convergence.

\begin{thm}[order $2$ of convergence for schemes of type I]\label{theorem-type-I}
Let $\dd$ be a function of class ${\rm C}^3([-1,1])$ such that $\dd(-1) = \dd(1) = 0$. Suppose that the sets of nodes and points satisfy the conditions stated in Section \ref{SECTION-mesh} and that $\beta_N(\dd)$ goes to zero at least with order $2$. That is to say, suppose that
\begin{align}
\label{interlacing-thm-I}
& -1 = \mu_{1/2} < \mu_1 < \mu_{1+1/2} < \cdots < \mu_{N-1/2} < \mu_N < \mu_{N+1/2} = 1,\\
\label{mesh-order-is-1-thm-I} & \tiM_N = O(N^{-1}),\\
\label{mid-point-1-hyp-thm-I} & D_N^\star = O(N^{-q})\ \mbox{with}\ q \geq 2,\\
\label{mid-point-2-hyp-thm-I} & D_N = O(N^{-r})\ \mbox{with}\ r \geq 2,\\
\label{hyp-min-points-thm-I} & \frac{1}{m_N^*} = O(N^s)\ \mbox{with}\ 1 \leq s  \leq 4m - 2,\ \mbox{where}\ m = \min\{q,r\},\\
\label{hyp-betaN-thm-I} & \beta_N(\dd) = O(N^{-t})\ \mbox{with}\ t \geq 2.
\end{align}

\par Then, the scheme (\ref{Type-I-scheme}) converges with order $2$ for any function $f$ of class ${\rm C}^5([-1,1])$, and the same is true if [$D_N = D_N^* = 0$ and the hypotheses (\ref{interlacing-thm-I}) and (\ref{mesh-order-is-1-thm-I}) hold] or if
[$d_1 = \cdots = d_{N-1}$ and the hypotheses (\ref{interlacing-thm-I})--(\ref{hyp-min-points-thm-I}) hold].
\end{thm}

\begin{proof}
Let $f$ be a function of class ${\rm C}^5([-1,1])$.

\par Thanks to Proposition \ref{BOUND-for-eps1-and-epsN} and Lemma \ref{mother-formula-2} we know that
\begin{equation}
\max \{|\varepsilon_1(f)|, |\varepsilon_N(f)|\} = O(N^{-2}).
\end{equation}

\par Moreover, since $C_N = D_N + D_N^* = O(N^{-m})$, Proposition \ref{BOUND-for-epsn} implies that
\begin{align}
\max_{2 \leq n \leq N-1} |\varepsilon_n^{(1)}(f)| & = O(N^{-\min\{r,t\}}),\\
\max_{2 \leq n \leq N-1} |\varepsilon_n^{(2)}(f)| & = O(N^{-2}),\\
\max_{2 \leq n \leq N-1} |\varepsilon_n^{(3)}(f)| & = O(N^{-\min\{r+2,t+2\}}),\\
\max_{2 \leq n \leq N-1} |\varepsilon_n^{(4)}(f)| & = O(N^{-\min\{3,4m-s\}}).
\end{align}

\par In summary,
\begin{equation}
\max_{1 \leq n \leq N} |\varepsilon_n(f)| = O(N^{-2}).
\end{equation}

\par On the other hand,
\begin{equation}
\max_{1 \leq n \leq N} |E_n^\star(\dd f')| = O(N^{-2})
\end{equation}
by Lemma \ref{mother-formula-1}, and so, in virtue of Proposition \ref{ER-type-I},
\begin{equation}
\max_{1 \leq n \leq N} |R_n(f)| = O(N^{-2}),
\end{equation}
which proves convergence of order $2$.

\par The last two statements follow by simple observation of the bounds in Proposition \ref{BOUND-for-epsn}, and of Definition \ref{betaNDEF}.
\end{proof}

\par Since $\dd$ is bounded, the hypothesis (\ref{hyp-betaN-thm-I}) is automatically satisfied if
\begin{equation}
\label{whynot}
\max_{2 \leq n \leq N-1} \left|\frac{d_{n-1} - d_n}{\mu_{n + 1/2} - \mu_{n - 1/2}}\right| = O(N^{-p})\ \mbox{with}\ p \geq 2,
\end{equation}
so the reader might wonder why we have not used this assumption in the previous theorem. After all, that way the set of hypotheses would be independent of $\dd$. The reason is that, for $\dd(\mu) = 1 - \mu^2$ and the choices of nodes and points made by Haldy and Ligou (Subsection \ref{subsec-HL}), condition (\ref{whynot}) is not satisfied, while (\ref{hyp-betaN-thm-I}) holds with $t = 2$.

\par To finish this subsection, let us comment that in Proposition \ref{BOUND-for-epsn}, and hence in Theorem \ref{theorem-type-I}, we can change $\beta_N(\dd)$ for
\begin{equation}
\tilde{\beta}_N(\dd) = \max_{2 \leq n \leq N-1} \left|\frac{(d_{n-1} - d_n) \dd(\mu_{n-1/2})}{\mu_{n + 1/2} - \mu_{n - 1/2}}\right|.
\end{equation}
Indeed, we could have used
\begin{multline}
\dd(\mu_{n+1/2}) f^{r+1)}(\mu_{n+1/2}) = \dd(\mu_{n-1/2}) f^{r+1)}(\mu_{n-1/2})\\
+ 2 h_n^* (\dd f^{r+1)})'(c^{(r)}_n)
\end{multline}
instead of Equation (\ref{Taylor}) in the proof of Proposition \ref{BOUND-for-epsn}, and thus eliminate the evaluations at $\mu_{n+1/2}$ to be left with the evaluations at $\mu_{n-1/2}$.

\subsection{Application of the theory to some examples. Numerical results}\label{subsection-application-type-I}
In the tables below, $E$ will denote the maximum of the absolute values of the errors in the complete set of nodes, i. e.,
\begin{equation}\label{DEFerror}
E = \max_{1 \leq n \leq N} |\Delta_{\rm FP} f(\mu_n) - \Delta_{{\rm FP},N} f(\mu_n)|.
\end{equation}
(changing obviously $\max_{1 \leq n \leq N}$ by $\max_{1 \leq n \leq 2N}$ when HR mode is used).

\paragraph{Lee's scheme} For this scheme, operated in FR mode,
\begin{itemize}
\item Hypothesis (\ref{interlacing-thm-I}) is obviously satisfied due to the definition of the points $\{\mu_{n+1/2}\}_{n = 0}^N$.

\item Hypothesis (\ref{mesh-order-is-1-thm-I}) is satisfied in virtue of Proposition \ref{GLnodesOK}.

\item Hypothesis (\ref{mid-point-1-hyp-thm-I}), with $q = 2$, is supported by the results in Section \ref{SECTION-properties-GL}, since
\begin{align}
& |\hmu_1 - \mu_1| = |\hmu_N - \mu_N| < \mu_2 + 1\quad \mbox{and}\\
& \hmu_n - \mu_n = \frac{1}{2} \left( \frac{\mu_{n-1} + \mu_{n+1}}{2} - \mu_n \right)\ \mbox{for}\ n \in \{2,\dots,N-1\}.
\end{align}

\item Hypothesis (\ref{mid-point-2-hyp-thm-I}) is obviously satisfied because $D_N = 0$.

\item Hypothesis (\ref{hyp-min-points-thm-I}) holds with $s = 2$ in virtue of Proposition \ref{GLnodesOK}, because
\begin{equation}
\frac{1}{\mu_2 + 1} < \frac{1}{m_N^*} < \frac{1}{\mu_1 + 1}.
\end{equation}

\item Hypothesis (\ref{hyp-betaN-thm-I}) is obviously satisfied because $\beta_N(\dd) = 0$.
\end{itemize}

\par According to Theorem \ref{theorem-type-I}, Lee's scheme in FR mode is expected to converge with order $2$. Table \ref{tableLeeFR} shows the numerical results got for the FP Laplacian of $f(\mu) = {\rm e}^\mu$. These results are in agreement with the theoretical prediction. Roundoff errors start spoiling the computations in the last row, where the order decays down to $1.53$.

\begin{table}[H]
\centering
\begin{tabular}{ccccc}
$N$ & $E$ & order & $q$ & $s$\\
\hline\hline\\[-2.9mm]
$50$ & $1.54 \times 10^{-2}$ & & &\\
$100$ & $3.96 \times 10^{-3}$ & $1.96$ & $1.98$ & $1.98$\\
$500$ & $1.61 \times 10^{-4}$ & $1.99$ & $1.99$ & $1.99$\\
$1000$ & $4.02 \times 10^{-5}$ & $2.00$ & $2.00$ & $2.00$\\
$5000$ & $1.61 \times 10^{-6}$ & $2.00$ & $2.00$ & $2.00$\\
$10000$ & $4.10 \times 10^{-7}$ & $1.97$ & $2.00$ & $2.00$\\
$20000$ & $1.42 \times 10^{-7}$ & $1.53$ & $2.00$ & $2.00$
\end{tabular}
\caption{Numerical results for Lee's scheme operated in FR mode. $f(\mu) = {\rm e}^\mu$, $\dd(\mu) = 1 - \mu^2$.} \label{tableLeeFR}
\end{table}

\par In HR mode, Lee's scheme behaves similarly, that is, converges with order $2$, but roundoff errors appear earlier, due to the extreme proximity of the nodes in the neighborhood of $0$.

\par This scheme converges with order $2$ if $\dd$ is any function of class ${\rm C}^3([-1,1])$ such that $\dd(-1) = \dd(1) = 0$. The reason is that $D_N = 0$ implies $\beta_N(\dd) = 0$.

\paragraph{Haldy\=/Ligou's scheme} For this scheme, operated in FR mode,
\begin{itemize}
\item Hypothesis (\ref{interlacing-thm-I}) is supported by the results in Section \ref{SECTION-properties-GL}.

\item Hypothesis (\ref{mesh-order-is-1-thm-I}) is satisfied in virtue of Proposition \ref{GLnodesOK}.

\item Hypotheses (\ref{mid-point-1-hyp-thm-I}) and (\ref{mid-point-2-hyp-thm-I}) hold with $q = r = 2$, which is again supported by the results in Section \ref{SECTION-properties-GL}.

\item Hypothesis (\ref{hyp-min-points-thm-I}) holds with $s = 2$ in virtue of Proposition \ref{GLweightsOK}, because
\begin{equation}
\frac{1}{m_N^*} = \frac{1}{w_1}.
\end{equation}

\item If $\dd(\mu) = 1 - \mu^2$, hypothesis (\ref{hyp-betaN-thm-I}) is satisfied with $t = 2$. This assertion is supported by some asymptotic analysis of the same type as that considered in \cite{LPSE23}. We observe that the number $\beta_N(\dd)$ can alternatively be written as
$$
\beta_N(\dd) = \max_{2\le n \le N-1} \left|\left(\frac{\mu_{n+1}-\mu_{n-1}}{2w_n}-1\right) \dd(\mu_{n+1/2})\right|,
$$
and that it is known (see \cite{LPSE23}, Theorem 1) that
$$
\Delta_k=\left|\left(\frac{\mu_{k+1}-\mu_{k-1}}{2w_k}-1\right)\right| = O(N^{-2})
$$
if $\mu_{k\pm 1}$ are in a fixed interval $[a,b]\subset(-1,1)$. Contrarily, when $k$
is fixed, it is known that $\Delta_k = O(1)$ (with a small error constant), but in that
case we have $\dd(\mu_{k+1/2}) = O(N^{-2})$.
\end{itemize}

\par The comments made in the previous example are valid for this one. Table \ref{tableHLFR} shows the numerical results, which corroborate that Haldy\=/Ligou's scheme in FR mode converges with order $2$.

\begin{table}[H]
\centering
\begin{tabular}{ccccccc}
$N$ & $E$ & order & $q$ & $r$ & $s$ & $t$\\
\hline\hline\\[-2.9mm]
$50$ & $8.68 \times 10^{-3}$ & & & & &\\
$100$ & $2.20 \times 10^{-3}$ & $1.98$ & $1.98$ & $1.98$ & $1.99$ & $1.99$\\
$500$ & $8.92 \times 10^{-5}$ & $1.99$ & $1.99$ & $1.99$ & $1.99$ & $2.00$\\
$1000$ & $2.23 \times 10^{-5}$ & $2.00$ & $2.00$ & $2.00$ & $2.00$ & $2.00$\\
$5000$ & $8.95 \times 10^{-7}$ & $2.00$ & $2.00$ & $2.00$ & $2.00$ & $2.00$\\
$10000$ & $2.31 \times 10^{-7}$ & $1.95$ & $2.00$ & $2.00$ & $2.00$ & $2.00$\\
$20000$ & $9.75 \times 10^{-8}$ & $1.24$ & $2.00$ & $2.00$ & $2.00$ & $2.00$
\end{tabular}
\caption{Numerical results for Haldy\=/Ligou's scheme operated in FR mode. $f(\mu) = {\rm e}^\mu$, $\dd(\mu) = 1 - \mu^2$.} \label{tableHLFR}
\end{table}

\par In HR mode, however, Haldy\=/Ligou's scheme is not convergent. This is shown in Table \ref{tableHLHR}. The `$t$' column tells us that the problem is that the hypothesis (\ref{hyp-betaN-thm-I}) is no longer satisfied. We have not included the `order' column since, in the absence of convergence, this value loses interest. On the other hand, Figure \ref{fig-HLHR} shows that it is at nodes close to $0$ where the scheme fails, which the reader can connect with the definition of $\beta_N(\dd)$ and the fact that $\dd$ is not zero at $0$, while points $\mu_{n+1/2}$ are accumulating quadratically on both sides of $0$.

\par The above mentioned accumulation of points $\mu_{n+1/2}$ around $0$ does not exist in FR mode, and the quadratic accumulation towards $-1$ and $1$ is not a problem, since there $\dd(\mu_{n+1/2})$ tends to zero at a rate that compensates for this accumulation and is enough for $\beta_N(\dd)$ to be a $O(N^{-2})$.

\begin{table}[H]
\centering
\begin{tabular}{cccccc}
$2N$ & $E$ & $q$ & $r$ & $s$ & $t$\\
\hline\hline\\[-2.9mm]
$50$ & $2.20 \times 10^{-1}$ & & & &\\
$100$ & $2.21 \times 10^{-1}$ & $1.97$ & $1.96$ & $1.97$ & $-1.14 \times 10^{-2}$\\
$500$ & $2.21 \times 10^{-1}$ & $1.99$ & $1.99$ & $1.99$ & $-1.61 \times 10^{-3}$\\
$1000$ & $2.21 \times 10^{-1}$ & $2.00$ & $2.00$ & $2.00$ & $-1.19 \times 10^{-4}$\\
$5000$ & $2.22 \times 10^{-1}$ & $2.00$ & $2.00$ & $2.00$ & $-1.64 \times 10^{-5}$\\
$10000$ & $2.28 \times 10^{-1}$ & $2.00$ & $2.00$ & $2.00$ & $-1.08 \times 10^{-6}$\\
$20000$ & $3.64 \times 10^{-1}$ & $2.00$ & $2.00$ & $2.00$ & $-4.54 \times 10^{-9}$
\end{tabular}
\caption{Numerical results showing that Haldy\=/Ligou's scheme operated in HR mode is not convergent. $f(\mu) = {\rm e}^\mu$, $\dd(\mu) = 1 - \mu^2$.} \label{tableHLHR}
\end{table}

\begin{figure}[ht]
    \centering
    \subfloat[$2N = 20$.]{
    \includegraphics[width = 0.4\textwidth, height = 0.3\textwidth] {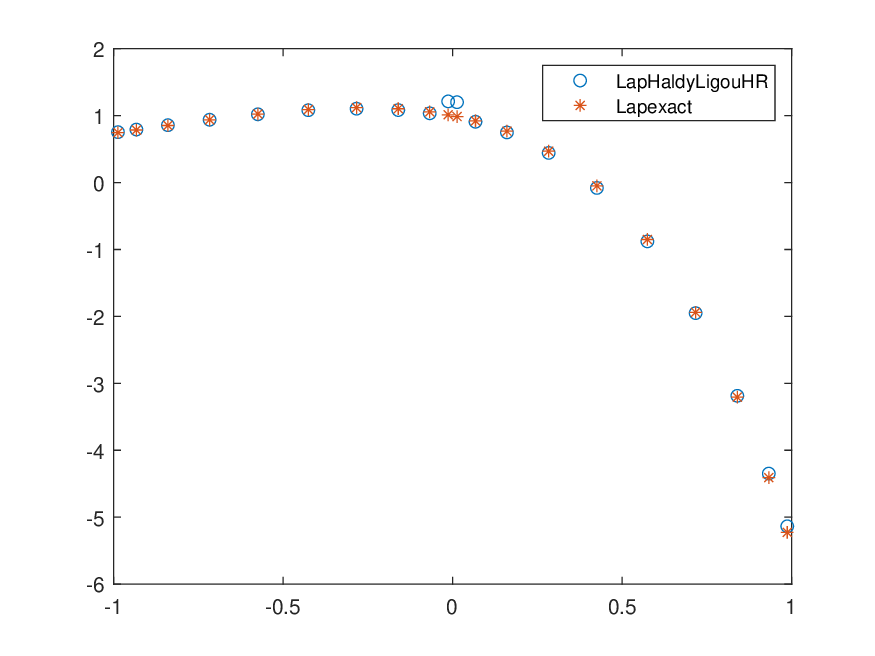}
    \label{fig-HLHR-2Neq20}
    }
    \subfloat[$2N = 100$.]{
    \includegraphics[width = 0.4\textwidth, height = 0.3\textwidth] {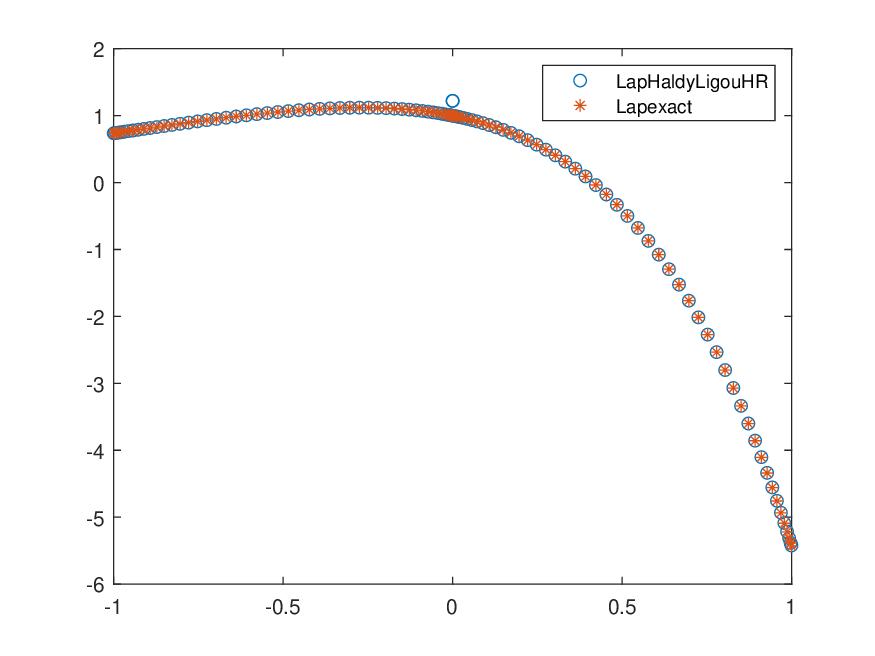}
    \label{fig-HLHR-2Neq100}
    }
    \caption{Haldy\=/Ligou's scheme in HR mode cannot compute good approximations of the FP Laplacian in the vicinity of $0$. $f(\mu) = {\rm e}^\mu$, $\dd(\mu) = 1 - \mu^2$.}
    \label{fig-HLHR}
\end{figure}

\paragraph{Uniform mesh}\label{par-uniform-mesh} (a non\=/GL scheme of type I and order $2$)
Let us take $h = 2/N$ and define
\begin{itemize}
\item $\mu_1 = -1 + h/2$, $\mu_{n+1} = \mu_n + h$ for $n = 1,\dots,N-1$,
\item $\mu_{1/2} = -1$, $\mu_{n+1/2} = (\mu_n + \mu_{n+1})/2$ for $n = 1,\dots,N-1$, and $\mu_{N+1/2} = 1$.
\end{itemize}

\par Then, by Theorem \ref{theorem-type-I}, the corresponding scheme of type I converges with order $2$, because $D_N^* = D_N = 0$ and the hypotheses (\ref{interlacing-thm-I}) and (\ref{mesh-order-is-1-thm-I}) are trivially met. Results are shown in Table \ref{tableUNIFORM-order-2}.

\begin{table}[H]
\centering
\begin{tabular}{ccc}
$N$ & $E$ & order\\
\hline\hline\\[-2.9mm]
$50$ & $2.44 \times 10^{-3}$ &\\
$100$ & $6.23 \times 10^{-4}$ & $1.97$\\
$500$ & $2.53 \times 10^{-5}$ & $1.99$\\
$1000$ & $6.33 \times 10^{-6}$ & $2.00$\\
$5000$ & $2.54 \times 10^{-7}$ & $2.00$\\
$10000$ & $6.34 \times 10^{-8}$ & $2.00$\\
$20000$ & $5.12 \times 10^{-8}$ & $3.09 \times 10^{-1}$
\end{tabular}
\caption{Numerical results for the scheme of type I of order $2$ on uniform mesh. $f(\mu) = {\rm e}^\mu$, $\dd(\mu) = 1 - \mu^2$.} \label{tableUNIFORM-order-2}
\end{table}

\par If we reinterpret the discrete zeroth and first moment properties of Definition \ref{DEF-discrete-moment} using the quadrature formula
\begin{multline}\label{new-quadrature}
\int_{-1}^1 G(\mu)\ d\mu = \int_{-1}^{\mu_1} G(\mu)\ d\mu + \int_{\mu_1}^{\mu_{N}} G(\mu)\ d\mu + \int_{\mu_{N}}^1 G(\mu)\ d\mu\\
\approx (h/2) G_1 + \mbox{(trapezoidal approximation)} + (h/2) G_N = h \sum_{n=1}^N G_n,
\end{multline}
understanding that $G_n$ is an approximation of $G(\mu_n)$, then this scheme satisfies both properties.

\par If the mesh $\{\mu_n\}_{n=1}^N$ is uniform, but the distance between $\mu_1$ and $-1$ or between $\mu_N$ and $1$ is different from $h/2$, then the scheme can easily stop being of order $2$. Table \ref{tableUNIFORM-order-1} shows that the order reduces to $1$ if $\mu_1 = -1 + 2/N$, $\mu_N = 1 - 1/N$ and $\{\mu_n\}_{n=2}^{N-1}$ are placed so that $\{\mu_n\}_{n=1}^N$ is uniform. The reason for the order drop is that now $D_N^*$ is only a $O(N^{-1})$ (i.~e., $q = 1$ in Table \ref{tableUNIFORM-order-1}).

\begin{table}[H]
\centering
\begin{tabular}{cccc}
$N$ & $E$ & order & $q$\\
\hline\hline\\[-2.9mm]
$50$ & $7.38 \times 10^{-3}$ & &\\
$100$ & $3.68 \times 10^{-3}$ & $1.00$ & $1.01$\\
$500$ & $7.36 \times 10^{-4}$ & $1.00$ & $1.00$\\
$1000$ & $3.68 \times 10^{-4}$ & $1.00$ & $1.00$\\
$5000$ & $7.36 \times 10^{-5}$ & $1.00$ & $1.00$\\
$10000$ & $3.68 \times 10^{-5}$ & $1.00$ & $1.00$\\
$20000$ & $1.84 \times 10^{-5}$ & $1.00$ & $1.00$
\end{tabular}
\caption{Numerical results for an instance of scheme of type I of order $1$ on uniform mesh. $f(\mu) = {\rm e}^\mu$, $\dd(\mu) = 1 - \mu^2$.} \label{tableUNIFORM-order-1}
\end{table}

\section{Schemes of type II}\label{SECTION-type-II}
We will call schemes of type II those schemes obtained by substituting in Equations (\ref{Type-I-scheme-a})--(\ref{Type-I-scheme-c}) the values $\{\dd(\mu_{n+1/2})\}_{n=0}^{N}$ by $\{\alpha_{n+1/2}\}_{n=0}^{N}$, being $\alpha_{1/2} = \dd(\mu_{1/2}) = 0$, and $\alpha_{n+1/2}$ a certain approximation of $\dd(\mu_{n+1/2})$ for $n = 1,\dots,N$. We notice that $\alpha_{N+1/2}$ can be $0$ or not.

\par So, these schemes are defined as follows:
\begin{equation}
\label{Type-II-scheme-a}
\Delta_{{\rm FP},N} f(\mu_1) = \frac{\alpha_{1+1/2} \frac{f(\mu_2) - f(\mu_1)}{\mu_2 - \mu_1}}{\mu_{1+1/2} + 1},
\end{equation}

\begin{eqnarray}
\nonumber & & \Delta_{{\rm FP},N} f(\mu_n) = \frac{\alpha_{n+1/2} \frac{f(\mu_{n+1}) - f(\mu_n)}{\mu_{n+1} - \mu_n} - \alpha_{n-1/2} \frac{f(\mu_n) - f(\mu_{n-1})}{\mu_n - \mu_{n-1}}}{\mu_{n+1/2} - \mu_{n-1/2}}\\
\label{Type-II-scheme-b} & & \mbox{for}\ n = 2,\dots,N-1,
\end{eqnarray}

\begin{equation}
\label{Type-II-scheme-c}
\Delta_{{\rm FP},N} f(\mu_N) = \frac{- \alpha_{N-1/2} \frac{f(\mu_N) - f(\mu_{N-1})}{\mu_N - \mu_{N-1}}}{1 - \mu_{N-1/2}}.
\end{equation}

\par Obviously, the family of schemes of type I is strictly contained in the family of schemes of type II.

\par After (\ref{Type-I-scheme}), when $\alpha_{N+1/2} = 0$ a scheme of type II can be written as
\begin{eqnarray}
\nonumber & & \Delta_{{\rm FP},N} f(\mu_n) = \frac{\alpha_{n+1/2} \frac{f(\mu_{n+1}) - f(\mu_n)}{\mu_{n+1} - \mu_n} - \alpha_{n-1/2} \frac{f(\mu_n) - f(\mu_{n-1})}{\mu_n - \mu_{n-1}}}{\mu_{n+1/2} - \mu_{n-1/2}}\\
\label{Type-II-scheme} & & \mbox{for}\ n = 1,\dots,N.
\end{eqnarray}

\par The values of $\alpha_{n+1/2}$ are computed from those of the nodes $\mu_n$ and the points $\mu_{n+1/2}$. Let us explain how this can be done. Notice that, having fixed $\alpha_{1/2} = 0$, there is only one way of choosing $\{\alpha_{n+1/2}\}_{n=1}^{N}$ that makes the scheme exact on $\PO_1$. Indeed, since it is obviously exact when $f$ is constant and $\Delta_{\rm FP} f(\mu) = \dd'(\mu)$ when $f(\mu) = \mu$, we conclude that a scheme of type II is exact on $\PO_1$ if, and only if,
\begin{equation}
\label{condition-on-alpha}
\frac{\alpha_{n+1/2} - \alpha_{n-1/2}}{\mu_{n+1/2} - \mu_{n-1/2}} = \dd'(\mu_n)\ \mbox{for}\ n = 1,\dots,N.
\end{equation}

\par Considering now that $\dd'(\mu_n) = -2 \mu_n$, we see that $\{\alpha_{n+1/2}\}_{n=0}^{N}$ must be defined by
\begin{align}
\label{alpha-DEF-1} & \alpha_{1/2} = \dd(\mu_{1/2}) = 0,\\
\label{alpha-DEF-2} & \alpha_{n+1/2} = \alpha_{n-1/2} - 2 \mu_n (\mu_{n+1/2} - \mu_{n-1/2})\ \mbox{for}\ n = 1,\dots,N
\end{align}
if we want exactness on $\PO_1$.

\begin{defi}\label{DEFlambdan}
For $n = 0,\dots,N$,
\begin{equation}
\lambda_n = \dd(\mu_{n+1/2}) - \alpha_{n+1/2},
\end{equation}
where it is understood that $\{\alpha_{n+1/2}\}_{n=0}^{N}$ is the set defined by (\ref{alpha-DEF-1})--(\ref{alpha-DEF-2}).
\end{defi}

\begin{defi}\label{DEFLambdaN}
$\Lambda_N = \max_{0\leq n \leq N} |\lambda_n|$.
\end{defi}

\par For simplicity, we have decided to use the notations $\alpha_{n+1/2}$, $\lambda_n$ and $\Lambda_N$, and not $\alpha_{n+1/2}(\dd),\dots$, even when all three depend on $\dd$. Later on, it will be useful to remember this fact.

\par The following result states precisely what we mean by saying that $\alpha_{n+1/2}$ is an approximation of $\dd(\mu_{n+1/2})$.

\begin{thm}\label{alpha-approximates-D}
Assume that the hypotheses (\ref{interlacing}), (\ref{mesh-order-is-1}), and (\ref{mid-point-1-hyp}) hold. Then
\begin{equation}
\Lambda_N = O(N^{-2}).
\end{equation}
\end{thm}

\begin{proof}
The initial value problem (IVP)
\begin{equation}
\left\{
\begin{array}{l}
y' = f(x),\ -1 < x < 1,\\
y(-1) = \eta \in \RE,
\end{array}
\right.
\end{equation}
with $f \in {\rm C}^2([-1,1])$, has got a unique solution $y \in {\rm C}^3([-1,1])$.

\par Since $\mu_{n+1/2} - \mu_{n-1/2} = 2 h_n^*$, the hypotheses (\ref{interlacing}) and (\ref{mesh-order-is-1}) guarantee that this IVP is solved with order $2$ of convergence by the numerical scheme
\begin{equation}
\left\{
\begin{array}{l}
y_0 = \eta,\\
y_n = y_{n-1} + 2 h_n^* f(\mu_{n-1/2} + h_n^*),\ n = 1,\dots,N,
\end{array}
\right.
\end{equation}
where $y_n$ represents an approximation of $y(\mu_{n+1/2})$.

\par The order $2$ of convergence is kept if we replace $f(\mu_{n-1/2} + h_n^*)$ by $f(\mu_{n-1/2} + h_{n-}^*) = f(\mu_n)$, because
\begin{equation}
\max_{1 \leq n \leq N}|h_n^* - h_{n-}^*| = \max_{1 \leq n \leq N} |d_n^*| = D_N^* = O(N^{-q})\ \mbox{with}\ q \geq 2
\end{equation}
in virtue of the hypothesis (\ref{mid-point-1-hyp}).

\par If we solve with the adapted scheme
\begin{equation}
\left\{
\begin{array}{l}
y_0 = \eta,\\
y_n = y_{n-1} + 2 h_n^* f(\mu_n),\ n = 1,\dots,N,
\end{array}
\right.
\end{equation}
the IVP determined by the data $f(x) = -2x$ and $\eta = 0$, the solution of which is $y(x) = \dd(x) = 1 - x^2$, we find that $y_n = \alpha_{n+1/2}$ for all $n = 0,\dots,N$, and the proof is done. Details are given in \cite{LPSE23}.
\end{proof}

\subsection{Example: Morel's scheme}
Morel's scheme, in FR mode, is the scheme of type II obtained when:
\begin{itemize}
\item Nodes and points are the same than in the Haldy\=/Ligou's scheme, that is, $\{\mu_n\}_{n=1}^N$ are the GL nodes, and the points $\{\mu_{n+1/2}\}_{n=0}^N$ are those defined by $\mu_{1/2} = -1$, $\mu_{n+1/2} = \mu_{n - 1/2} + w_n$ for $n = 1,\dots,N$, being $\{w_n\}_{n=1}^N$ the GL weights.
\item $\{\alpha_{n+1/2}\}_{n=0}^N$ are the values defined by $\alpha_{1/2} = 0$, $\alpha_{n+1/2} = \alpha_{n-1/2} - 2 \mu_n w_n$ for $n = 1,\dots,N$. Since
\begin{equation}
\sum_{n=1}^{N} w_n \mu_n = \int_{-1}^1 \mu\ d\mu = 0,
\end{equation}
it turns out that these values are symmetric with respect to $0$:
\begin{equation}\label{alpha-symmetric}
\alpha_{n+1/2} = \alpha_{N-n+1/2}\ \mbox{for}\ n = 0,\dots,N.
\end{equation}
In particular, $\alpha_{N+1/2} = 0$.
\end{itemize}

\par So this scheme reads as follows:
\begin{eqnarray}
\nonumber & & \Delta_{{\rm FP},N} f(\mu_n) = \frac{\alpha_{n+1/2} \frac{f(\mu_{n+1}) - f(\mu_n)}{\mu_{n+1} - \mu_n} - \alpha_{n-1/2} \frac{f(\mu_n) - f(\mu_{n-1})}{\mu_n - \mu_{n-1}}}{w_n}\\
\label{Morel-scheme} & & \mbox{for}\ n = 1,\dots,N.
\end{eqnarray}

\par Morel introduced it in \cite{MO85} expressly so that the discrete zeroth and first moment properties were fulfilled. References \cite{KITR08}, \cite{OLFR10}, \cite{PWP20}, and \cite{WAPR12} provide examples of its application.

\par As seen in Equation (\ref{Morel-scheme}), the points $\mu_{n+1/2}$ are not needed for describing this scheme, and in fact Morel did not mention them at all in \cite{MO85}. However, it is not possible to carry out a convergence analysis without taking these points into account.

\par Both the discrete zeroth and first moment properties are still satisfied when it is used in HR mode, but this is completely irrelevant, since, as it happens to Haldy\=/Ligou's, this scheme in HR mode does not converge.

\begin{remark}[other choices of $\alpha_{n+1/2}$ can be made]
When compared to having order $2$, having exactness on $\PO_1$ is not that important (schemes of type I are not exact on $\PO_1$ unless $\mu_n$ be the mid\=/point of the cell $[\mu_{n-1/2},\mu_{n+1/2}]$), but presenting the problem of calculating $\alpha_{n+1/2}$ from the exactness on $\PO_1$ has two advantages: on the one hand, it serves as a mnemonic rule to remember how to calculate $\alpha_{n+1/2}$ even when a diffusivity other than $\dd$ is used; on the other, it coincides with the choice that Morel made, following a different path, for his scheme.

\par Having said that, and observing the proof of Theorem \ref{alpha-approximates-D}, we could modify the values of $\alpha_{n+1/2}$ simply by using a different numerical method from the one used in this proof. To have an instance, let us suppose that $\{\mu_n\}_{n=1}^N$ and $\{\mu_{n+1/2}\}_{n=0}^N$ are those of Morel's scheme. Then, the choice
\begin{align}
\label{alpha-RK4-DEF-1} & \alpha_{1/2} = 0,\\
\label{alpha-RK4-DEF-2} & \alpha_{n+1/2} = \alpha_{n-1/2} - w_n (2\mu_{n-1/2} + w_n)\ \mbox{for}\ n = 1,\dots,N,
\end{align}
which results from solving the IVP in the proof of Theorem \ref{alpha-approximates-D} with the classic Runge\=/Kutta method of fourth order, provides us with values satisfying condition (\ref{alpha-symmetric}) and with a convergent method of experimental order $2$. However, the discrete first moment property ceases to be met.
\end{remark}

\subsection{Analysis of convergence}
We are going to analyze schemes of type II only for $\dd(\mu) = 1 - \mu^2$ and restricting ourselves to the case in which the numbers $\{\alpha_{n+1/2}\}_{n=0}^N$ are given by (\ref{alpha-DEF-1})--(\ref{alpha-DEF-2}). Therefore, we can use Theorem \ref{alpha-approximates-D}. Thanks to the fact that we have already analyzed the convergence of schemes of type I, the task ahead will not be so complicated.

\par Let us start with a useful lemma.
\begin{lemma}
\label{lemma-lambdan}
If $\{\lambda_n\}_{n=0}^N$ is given by Definition \ref{DEFlambdan}, then, for $n = 1,\dots,N$,
\begin{equation}
\lambda_n = \lambda_{n-1} - 2 d_n^* (\mu_{n+1/2} - \mu_{n-1/2}).
\end{equation}
\end{lemma}

\begin{proof}
By the definitions of $d_n^*$ and $\hmu_n$,
\begin{equation}
2 d_n^* = 2 \hmu_n - 2 \mu_n = \mu_{n-1/2} + \mu_{n+1/2} - 2 \mu_n.
\end{equation}

\par So,
\begin{multline}
\lambda_n = \lambda_{n-1} - 2 d_n^* (\mu_{n+1/2} - \mu_{n-1/2})\\
\Leftrightarrow \lambda_n = \lambda_{n-1} + \{2 \mu_n - (\mu_{n-1/2} + \mu_{n+1/2})\} (\mu_{n+1/2} - \mu_{n-1/2})\\
\Leftrightarrow \lambda_n = \lambda_{n-1} + 2 \mu_n (\mu_{n+1/2} - \mu_{n-1/2}) - \{(\mu_{n+1/2})^2 - (\mu_{n-1/2})^2\}.
\end{multline}

\par Noticing now that $(\mu_{n+1/2})^2 - (\mu_{n-1/2})^2 = \dd(\mu_{n-1/2}) - \dd(\mu_{n+1/2})$, one has
\begin{multline}
\lambda_n = \lambda_{n-1} - 2 d_n^* (\mu_{n+1/2} - \mu_{n-1/2})\\
\Leftrightarrow \dd(\mu_{n+1/2}) - \alpha_{n+1/2} = \dd(\mu_{n-1/2}) - \alpha_{n-1/2} + 2 \mu_n (\mu_{n+1/2} - \mu_{n-1/2})\\
- \dd(\mu_{n-1/2}) + \dd(\mu_{n+1/2}) \Leftrightarrow \alpha_{n+1/2} = \alpha_{n-1/2} - 2 \mu_n (\mu_{n+1/2} - \mu_{n-1/2}),
\end{multline}
which ends the proof, as the last equality is known to be true.
\end{proof}

\par The basic idea in this section is to use Lemma \ref{lemma-lambdan} to recast the scheme as a perturbation of a scheme of type I. Having done that, Theorem \ref{theorem-type-I} solves much of the problem.

\begin{prop}[error representation for schemes of type II]
\label{ER-type-II} Suppose that $f \in {\rm C}^2([-1,1])$ and that $\Delta_{{\rm FP},N} f(\mu_n)$ is defined by Equations (\ref{Type-II-scheme-a})--(\ref{Type-II-scheme-c}), with $\{\alpha_{n+1/2}\}_{n=0}^N$ given by Equations (\ref{alpha-DEF-1})--(\ref{alpha-DEF-2}). Then, for $n = 1,\dots,N$,
\begin{equation}
\Delta_{{\rm FP},N} f(\mu_n) = \Delta_{\rm FP} f(\mu_n) - \{R_n(f) + R_n^*(f)\},
\end{equation}
with $R_n(f)$ defined by Equation (\ref{DEFRn}) and
\begin{align}
\label{R1star} R_1^*(f) & = 2 d_1^* \{E_1(f) - f'(\mu_{1+1/2})\},\\
\nonumber R_n^*(f) & = \frac{\lambda_n}{2 h_n^*} \{ E_{n-1}(f) - E_n(f) + f'(\mu_{n+1/2}) - f'(\mu_{n-1/2})\}\\
\label{Rnstar} & \hspace{7mm} + 2 d_n^* \{E_{n-1}(f) - f'(\mu_{n-1/2})\}\ \mbox{for}\ n = 2,\dots, N-1,\\
\label{RNstar} R_N^*(f) & = \left( 2 d_N^* - \frac{\alpha_{N+1/2}}{1 - \mu_{N-1/2}} \right) \{E_{N-1}(f) - f'(\mu_{N-1/2})\}.
\end{align}
\par In the expressions above, $E_n(f)$ is that defined by Equation (\ref{EnDEF}).
\end{prop}

\begin{proof}
Let us distinguish the three possible cases.
\begin{itemize}
\item Case $n = 1$: use Definition \ref{DEFlambdan} and Lemma \ref{lemma-lambdan} to see that
\begin{equation}
\alpha_{1+1/2} = \dd(\mu_{1+1/2}) - \lambda_1 = \dd(\mu_{1+1/2}) + 2 d_1^* (\mu_{1+1/2} + 1).
\end{equation}

\par Then,
\begin{multline}
\Delta_{{\rm FP},N} f(\mu_1) = \frac{\alpha_{1+1/2} \frac{f(\mu_2) - f(\mu_1)}{\mu_2 - \mu_1}}{\mu_{1+1/2} + 1}\\
= \frac{\dd(\mu_{1+1/2}) \frac{f(\mu_2) - f(\mu_1)}{\mu_2 - \mu_1}}{\mu_{1+1/2} + 1} + 2 d_1^* \frac{f(\mu_2) - f(\mu_1)}{\mu_2 - \mu_1}\\
= \frac{\dd(\mu_{1+1/2}) \frac{f(\mu_2) - f(\mu_1)}{\mu_2 - \mu_1}}{\mu_{1+1/2} + 1} + 2 d_1^* \{f'(\mu_{1+1/2}) - E_1(f)\}.
\end{multline}

\par Lemma \ref{mother-formula-2} has been used in the last step above.

\par In other words,
\begin{equation}
\Delta_{{\rm FP},N} f(\mu_1) = \Delta_{\rm FP} f(\mu_1) - \{R_1(f) + R_1^*(f)\},
\end{equation}
with $R_1^*(f)$ defined by (\ref{R1star}).

\item Case $n \in \{2,\dots,N-1\}$: since, by Definition \ref{DEFlambdan},
\begin{equation}
\alpha_{n+1/2} = \dd(\mu_{n+1/2}) - \lambda_n\ \mbox{and}\ \alpha_{n-1/2} = \dd(\mu_{n-1/2}) - \lambda_{n-1},
\end{equation}
we have
\begin{multline}
\Delta_{{\rm FP},N} f(\mu_n) = \frac{\alpha_{n+1/2} \frac{f(\mu_{n+1}) - f(\mu_n)}{\mu_{n+1} - \mu_n} - \alpha_{n-1/2} \frac{f(\mu_n) - f(\mu_{n-1})}{\mu_n - \mu_{n-1}}}{\mu_{n+1/2} - \mu_{n-1/2}}\\
= \frac{\dd(\mu_{n+1/2}) \frac{f(\mu_{n+1}) - f(\mu_n)}{\mu_{n+1} - \mu_n} - \dd(\mu_{n-1/2}) \frac{f(\mu_n) - f(\mu_{n-1})}{\mu_n - \mu_{n-1}}}{\mu_{n+1/2} - \mu_{n-1/2}}\\
- \frac{\lambda_n \frac{f(\mu_{n+1}) - f(\mu_n)}{\mu_{n+1} - \mu_n} - \lambda_{n-1} \frac{f(\mu_n) - f(\mu_{n-1})}{\mu_n - \mu_{n-1}}}{\mu_{n+1/2} - \mu_{n-1/2}}\\
= \Delta_{\rm FP} f(\mu_n) - R_n(f) - \frac{\lambda_n \frac{f(\mu_{n+1}) - f(\mu_n)}{\mu_{n+1} - \mu_n} - \lambda_{n-1} \frac{f(\mu_n) - f(\mu_{n-1})}{\mu_n - \mu_{n-1}}}{\mu_{n+1/2} - \mu_{n-1/2}}.
\end{multline}
\par Now use $\lambda_{n-1} = \lambda_n + 2 d_n^* (\mu_{n+1/2} - \mu_{n-1/2})$, which holds by Lemma \ref{lemma-lambdan}, $\mu_{n+1/2} - \mu_{n-1/2} = 2 h_n^*$, and finally Lemma \ref{mother-formula-2} to get
\begin{equation}
\Delta_{{\rm FP},N} f(\mu_n) = \Delta_{\rm FP} f(\mu_1) - \{R_n(f) + R_n^*(f)\},
\end{equation}
with $R_n^*(f)$ defined by (\ref{Rnstar}).

\item Case $n = N$: Equation (\ref{DEFRn}) implies
\begin{equation}
\frac{-\dd(\mu_{N-1/2}) \frac{f(\mu_N) - f(\mu_{N-1})}{\mu_N - \mu_{N-1}}}{1 - \mu_{N-1/2}} = \Delta_{\rm FP} f(\mu_N) - R_N(f),
\end{equation}
which, used in combination with the three identities
\begin{equation}
\Delta_{{\rm FP},N} f(\mu_N) = \frac{-\alpha_{N-1/2} \frac{f(\mu_N) - f(\mu_{N-1})}{\mu_N - \mu_{N-1}}}{1 - \mu_{N-1/2}},
\end{equation}
\begin{multline}
\alpha_{N-1/2} = \dd(\mu_{N-1/2}) - \lambda_{N-1}\\
= \dd(\mu_{N-1/2}) - \{\lambda_N + 2d_N^*(1 - \mu_{N-1/2})\},
\end{multline}
and
\begin{equation}
\lambda_N = \dd(\mu_{N+1/2}) - \alpha_{N+1/2} = - \alpha_{N+1/2},
\end{equation}
ends the proof of this case.
\end{itemize}
\end{proof}

\par To properly understand the notation used in the following definition, recall that $\lambda_n$ depends on $\dd$.

\begin{defi}
$\beta_N^*(\dd)$ is the number defined by
\begin{equation}
\beta_N^*(\dd) = \max_{2 \leq n \leq N-1} \left| \frac{(d_{n-1} - d_n) \lambda_n}{\mu_{n+1/2} - \mu_{n-1/2}} \right|
\end{equation}
or, equivalently,
\begin{equation}
\beta_N^*(\dd) = \max_{2 \leq n \leq N-1} \left| \frac{(d_{n-1} - d_n) \lambda_n}{2 h_n^*} \right|.
\end{equation}
\end{defi}

\begin{defi}
$X_N = |\alpha_{N+1/2}/(1 - \mu_{N-1/2})|$.
\end{defi}

\par We are now in a position to prove the main result in this section.

\begin{thm}[order $2$ of convergence for schemes of type II]\label{theorem-type-II}
Suppose that
\begin{align}
\label{interlacing-thm-II}
& -1 = \mu_{1/2} < \mu_1 < \mu_{1+1/2} < \cdots < \mu_{N-1/2} < \mu_N < \mu_{N+1/2} = 1,\\
\label{mesh-order-is-1-thm-II} & \tiM_N = O(N^{-1}),\\
\label{mid-point-1-hyp-thm-II} & D_N^\star = O(N^{-q})\ \mbox{with}\ q \geq 2,\\
\label{mid-point-2-hyp-thm-II} & D_N = O(N^{-r})\ \mbox{with}\ r \geq 2,\\
\label{hyp-min-points-thm-II} & \frac{1}{m_N^*} = O(N^s)\ \mbox{with}\ 1 \leq s  \leq 4m - 2,\ \mbox{where}\ m = \min\{q,r\},\\
\label{hyp-betaN-thm-II} & \beta_N(\dd) = O(N^{-t})\ \mbox{with}\ t \geq 2,\\
\label{new-hyp-1} & \beta_N^*(\dd) = O(N^{-u})\ \mbox{with}\ u \geq 2,\\
\label{new-hyp-2} & X_N = O(N^{-v})\ \mbox{with}\ v \geq 2.
\end{align}

\par Then, the scheme (\ref{Type-II-scheme-a})--(\ref{Type-II-scheme-c}), with $\{\alpha_{n+1/2}\}_{n=0}^N$ given by (\ref{alpha-DEF-1})--(\ref{alpha-DEF-2}), converges with order $2$ for any function $f$ of class ${\rm C}^5([-1,1])$, and the same is true if [$D_N = D_N^* = 0$ and the hypotheses (\ref{interlacing-thm-II}), (\ref{mesh-order-is-1-thm-II}) and (\ref{new-hyp-2}) hold], if
[$d_1 = \cdots = d_{N-1}$ and the hypotheses (\ref{interlacing-thm-II})--(\ref{hyp-min-points-thm-II}) and (\ref{new-hyp-2}) hold] or if [$\Lambda_N = 0$ and the hypotheses (\ref{interlacing-thm-II})--(\ref{hyp-betaN-thm-II}) hold].

\par Furthermore, if $s \leq r$, where $r$ and $s$ are those in (\ref{mid-point-2-hyp-thm-II}) and (\ref{hyp-min-points-thm-II}), the hypothesis (\ref{new-hyp-1}) can be ignored, since it will be automatically fulfilled as a consequence of the others.
\end{thm}

\begin{proof}
Let us start by noticing that $\Lambda_N = \max_{0\leq n \leq N} |\lambda_n| = O(N^{-2})$, in virtue of Theorem \ref{alpha-approximates-D}. Then,
\begin{equation}
\beta_N^*(\dd) \leq \frac{2 D_N \Lambda_N}{m_N^*} = O(N^{s-r-2}),
\end{equation}
and so the hypothesis (\ref{new-hyp-1}) will indeed be automatically fulfilled if $s \leq r$.

\par Now we will prove the main part of the theorem. Thanks to Proposition \ref{ER-type-II}, we only need to prove that $\max_{1 \leq n \leq N} |R_n^*(f)| = O(N^{-2})$, because we already know that $\max_{1 \leq n \leq N} |R_n(f)| = O(N^{-2})$ by Theorem \ref{theorem-type-I}.

\par Notice that
\begin{equation}
\max_{1 \leq n \leq N-1} |E_n(f)| = O(N^{-2})
\end{equation}
by Lemma \ref{mother-formula-2}.

\begin{itemize}
\item Bound for $|R_1^*(f)|$:
\begin{multline}
|R_1^*(f)| = |2 d_1^* \{E_1(f) - f'(\mu_{1+1/2})\}|\\
\leq 2 D_N^*\left\{\left(\max_{1 \leq n \leq N-1} |E_n(f)|\right) + \|f'\|_\infty)\right\} = O(N^{-q}).
\end{multline}

\item Bound for $\max_{2 \leq n \leq N-1}|R_n^*(f)|$: let us fix $n \in \{2,\dots,N-1\}$ and understand that $\beta_N^* = \beta_N^*(\dd)$. We know from Equation (\ref{Rnstar}) that
\begin{multline}
R_n^*(f) = \frac{\lambda_n}{2 h_n^*} \{ E_{n-1}(f) - E_n(f) + f'(\mu_{n+1/2}) - f'(\mu_{n-1/2})\}\\
+ 2 d_n^* \{E_{n-1}(f) - f'(\mu_{n-1/2})\}\ \mbox{for}\ n = 2,\dots, N-1.
\end{multline}

\par Two parts of the expression above can be easily bounded:
\begin{multline}
|2 d_n^* \{E_{n-1}(f) - f'(\mu_{n-1/2})\}|\\
\leq 2 D_N^*\left\{\left(\max_{1 \leq n \leq N-1} |E_n(f)|\right) + \|f'\|_\infty)\right\} = O(N^{-q})
\end{multline}
and
\begin{multline}
\left| \frac{\lambda_n}{2 h_n^*} \{ f'(\mu_{n+1/2}) - f'(\mu_{n-1/2})\} \right|\\
= \left| \frac{\lambda_n}{2 h_n^*} \{ f'(\mu_{n-1/2}) + 2 h_n^* f''(c_n) - f'(\mu_{n-1/2})\} \right|\\
= | \lambda_n f''(c_n) | \leq \Lambda_N \|f''\|_\infty = O(N^{-2}).
\end{multline}

\par Finding a bound for
\begin{equation}
\left| \frac{\lambda_n}{2 h_n^*} \{ E_{n-1}(f) - E_n(f) \} \right|
\end{equation}
is in principle more difficult, but, introducing the definition
\begin{equation}
\tep_n(f) = \frac{\lambda_n}{2 h_n^*} \{ E_{n-1}(f) - E_n(f) \},
\end{equation}
noting the resemblance of $\tep_n(f)$ to $\varepsilon_n(f)$ in Equation (\ref{epsnDEF}), and using the same ideas than those in the proof of Proposition \ref{BOUND-for-epsn} (with $\dd \equiv 1$), one gets
\begin{equation}
\tep_n(f) = \tep_n^{(1)}(f) + \tep_n^{(2)}(f) + \tep_n^{(3)}(f) + \tep_n^{(4)}(f),
\end{equation}
with
\begin{align}
|\tep_n^{(1)}(f)| & \leq \beta_N^* \|f''\|_\infty + D_N \Lambda_N \|f'''\|_\infty,\\
\nonumber |\tep_n^{(2)}(f)| & \leq \frac{\beta_N^* (3 D_N + C_N)}{3} \|f'''\|_\infty + \frac{C_N \Lambda_N}{3} \|f'''\|_\infty\\
& \hspace{30mm} + \frac{(M_N^2 + 12 D_N^2)\Lambda_N}{24} \|f^{4)}\|_\infty,\\
\nonumber |\tep_n^{(3)}(f)| & \leq \frac{\beta_N^* \{(M_N^*)^2 + 8 C_N D_N + 4 C_N^2 + 12 D_N^2\}}{24} \|f^{4)}\|_\infty\\
& \hspace{5mm} + \frac{C_N D_N \Lambda_N}{3} \|f^{4)}\|_\infty + \frac{(M_N^2 D_N + 4 D_N^3)\Lambda_N}{24} \|f^{5)}\|_\infty,\\
|\tep_n^{(4)}(f)| & \leq \frac{Z_N \Lambda_N}{960} \|f^{5)}\|_\infty,
\end{align}
where $Z_N$ is given by Equation (\ref{DEFZN}).

\par Now, recalling that $C_N = D_N + D_N^* = O(N^{-m})$, it is clear that
\begin{equation}
\max_{2 \leq n \leq N-1}|R_n^*(f)| = O(N^{-2}).
\end{equation}

\item Bound for $|R_N^*(f)|$:
\begin{multline}
|R_N^*(f)| = \left|\left( 2 d_N^* - \frac{\alpha_{N+1/2}}{1 - \mu_{N-1/2}} \right) \{E_{N-1}(f) - f'(\mu_{N-1/2})\}\right|\\
\leq \left( 2 D_N^* + X_N \right) \left\{\left(\max_{1 \leq n \leq N-1} |E_n(f)|\right) + \|f'\|_\infty)\right\}\\
= O(N^{-\min\{q,v\}}).
\end{multline}
\end{itemize}
\par In summary, if $f \in {\rm C}^5([-1,1])$,
\begin{equation}
\max_{1 \leq n \leq N} |R_n(f) + R_n^*(f)| = O(N^{-2}),
\end{equation}
and so the scheme converges with order $2$. The statements that remain to be proved follow easily.
\end{proof}

\begin{remark}
Theorem \ref{theorem-type-I} (for $\dd(\mu) = 1 - \mu^2$) becomes a particular case of Theorem \ref{theorem-type-II}: the one that results from considering $\Lambda_N = 0$.
\end{remark}

\subsection{Application of the theory to some examples. Numerical results}
Recalling Equation (\ref{DEFerror}), $E$ will denote the maximum of the absolute values of the errors in the complete set of nodes.

\paragraph{Morel's scheme} For this scheme, operated in FR mode, the hypotheses (\ref{interlacing-thm-II})--(\ref{hyp-betaN-thm-II}) are met; the justifications given for the Haldy\=/Ligou's scheme are also valid for this one. Moreover, as $r = s$, the hypothesis (\ref{new-hyp-1}) is automatically satisfied, while the last hypothesis (\ref{new-hyp-2}) also holds because $\alpha_{N+1/2} = 0$.

\par Thus, according to Theorem \ref{theorem-type-II}, Morel's scheme in FR mode is expected to converge with order $2$. Numerical results in agreement with the theoretical prediction are displayed in Table \ref{tableMorelFR}, the rows of which stop at the moment where roundoff errors start to spoil the approximation.

\begin{table}[H]
\centering
\begin{tabular}{cccccccc}
$N$ & $E$ & order & $q$ & $r$ & $s$ & $t$ & $u$\\
\hline\hline\\[-2.9mm]
$50$ & $6.94 \times 10^{-3}$ & & & & &\\
$100$ & $1.76 \times 10^{-3}$ & $1.98$ & $1.98$ & $1.98$ & $1.99$ & $1.99$ & $3.97$\\
$500$ & $7.14 \times 10^{-5}$ & $1.99$ & $1.99$ & $1.99$ & $1.99$ & $2.00$ & $3.99$\\
$1000$ & $1.79 \times 10^{-5}$ & $2.00$ & $2.00$ & $2.00$ & $2.00$ & $2.00$ & $4.00$\\
$5000$ & $7.16 \times 10^{-7}$ & $2.00$ & $2.00$ & $2.00$ & $2.00$ & $2.00$ & $4.00$\\
$10000$ & $1.86 \times 10^{-7}$ & $1.94$ & $2.00$ & $2.00$ & $2.00$ & $2.00$ & $4.00$\\
$20000$ & $8.64 \times 10^{-8}$ & $1.11$ & $2.00$ & $2.00$ & $2.00$ & $2.00$ & $4.00$
\end{tabular}
\caption{Numerical results for Morel's scheme operated in FR mode. $f(\mu) = {\rm e}^\mu$, $\dd(\mu) = 1 - \mu^2$.} \label{tableMorelFR}
\end{table}

\par Like Haldy\=/Ligou's scheme, Morel's does not converge when used in HR mode, and the reason is the same: the hypothesis (\ref{hyp-betaN-thm-II}) is not fulfilled. This explains why in reference \cite{GALP23} the authors had to discard the use of the HR mode, and use the FR mode instead, when solving the FPE.\footnote{Excerpted from \cite{GALP23}: `After experimentation, our choice will be FRLGQ. Apparently, HRLGQ, while appropriate for the neutron transport, performs poorly for the FPE.'} The numerical and graphical results are very similar to those of Haldy\=/Ligou's and are omitted.

\begin{remark}
The fact that $\Lambda_N = O(N^{-2})$, which according to Theorem \ref{alpha-approximates-D} is true under the assumptions (\ref{interlacing-thm-II})--(\ref{mid-point-1-hyp-thm-II}) in Theorem \ref{theorem-type-II}, is also supported by strong theoretical evidence in \cite{LPSE23}.
\end{remark}

\paragraph{Uniform mesh} (a non\=/GL scheme of type II and order $2$) Let us take, as $\{\mu_n\}_{n=1}^N$ and $\{\mu_{n+1/2}\}_{n=0}^N$, the uniformly spaced sets that we took when defining the scheme of type I and of order $2$ on Subsection \ref{subsection-application-type-I} (uniform mesh). Then, the corresponding scheme of type II satisfies $D_N = D_N^* = \alpha_{N+1/2} = 0$, which implies convergence of order $2$ according to Theorem \ref{theorem-type-II}. In fact, it can be easily checked that in this case ones gets $\alpha_{n+1/2} = \dd(\mu_{n+1/2})$ for all $n = 0,\dots,N$, and so this scheme is exactly the scheme of type I described on Subsection \ref{subsection-application-type-I} (uniform mesh).

\section{Conclusions}\label{SECTION-conclusions}
Widely recognized difference schemes for discretizing the FP angular diffusion operator have been incorporated into a comprehensive framework, which has undergone thorough analysis. This analysis has allowed us to derive sets of sufficient conditions that guarantee the convergence with second\=/order accuracy for the schemes falling into the two categories defined in this work: type I and type II schemes.

\par By applying these general results, the study provides theoretical evidence supporting second\=/order convergence of Lee's, Hal\-dy\=/Li\-gou's, and Morel's schemes when they are operated in FR mode. Moreover, the study highlights that Hal\-dy\=/Li\-gou's and Morel's schemes do not exhibit convergence when operated in HR mode, which aligns with experimental observations documented in \cite{GALP23}. This finding holds significant implications and should be taken into consideration when solving the FPE.

\par Lastly, this research uncovers new properties of GL nodes and weights. The analysis of these properties, which necessitates the use of specialized techniques, is conducted in \cite{LPSE23}.

\section*{Funding}
OLP acknowledges support from Ministerio de Ciencia e Innovaci\'on, pro\-ject PID2021-122625OB-I00 with funds from
\begin{center}
MCIN/AEI/10.13039/501100011033/ ERDF, UE,
\end{center}
and from the Xunta de Galicia (2021 GRC Gl-1563 - ED431C 2021/15).

\par JS acknowledges support from Ministerio de Ciencia e Innovaci\'on, project
 PID2021-127252NB-I00 with funds from
 \begin{center}
MCIN/AEI/10.13039/501100011033/ FEDER, UE.
\end{center}

\section*{Acknowledgments}
The authors are grateful to Prof. Barry Ganapol from the Aerospace and Mechanical Department at the University of Arizona for his interest in this work and helpful advice after carefully reviewing parts of the paper.

\end{document}